\documentclass{article}
\usepackage[utf8]{inputenc}
\usepackage{amsthm}
\usepackage{amsmath}
\usepackage{mathrsfs}
\usepackage{amssymb}
\usepackage{graphicx}
\usepackage{tensor}
\usepackage{stackrel}
\usepackage{eulervm}
\usepackage{tikz-cd}
\tikzset{%
    symbol/.style={%
        draw=none,
        every to/.append style={%
            edge node={node [sloped, allow upside down, auto=false]{$#1$}}}
    }
}
\usepackage{enumitem}   
\makeatletter

\newcommand{\clim}{\operatorname*{lim}}
\newcommand{\colim}{\operatorname*{colim}}
\newcommand{\holim}{\operatorname*{holim}}
\newcommand{\coker}{\operatorname{coker}}
\newcommand{\mcl}{\mathcal}
\newcommand{\uhom}{\underline{Hom}}
\newcommand{\lra}{\longrightarrow}
\newcommand{\scr}{\mathscr}

\newcommand{\ot}{\otimes}

\newcommand{\ubb}{\underline{\mathbb{R}}}
\newcommand{\rdh}{\mathbb{R}\underline{Hom}}

\newcommand{\wot}{\widehat{\otimes}}

\newtheorem{theorem}{Theorem}[section]
\newtheorem{lemma}[theorem]{Lemma}
\newtheorem*{theorem*}{Theorem}
\newtheorem{proposition}[theorem]{Proposition}
\newtheorem{corollary}[theorem]{Corollary}

\theoremstyle{definition}
\newtheorem{definition}[theorem]{Definition}
\newtheorem{remark}[theorem]{Remark}
\newtheorem{example}[theorem]{Example}
\numberwithin{equation}{subsection}
\numberwithin{theorem}{subsection}

\title{Model Categories and the Higher Riemann-Hilbert Correspondence}
\author{Callum Galvin}

\begin{document}

\maketitle{}

\begin{abstract}

We construct a new model structure on the category of dg presheaves over a topological space $X$, obtained through the right Bousfield localization of the local projective model structure. The motivation for this construction arises from the study of the homotopy theory underlying higher Riemann-Hilbert correspondence theorems, as developed by Chuang, Holstein, and Lazarev.

Let $X$ be a smooth manifold. We prove the existence of a zig-zag of Quillen equivalences between the category of dg modules over the de Rham algebra and the category of dg presheaves of vector spaces over $X$. In the case where $X$ is a complex manifold, we obtain an analogous result, where the de Rham algebra is replaced by the Dolbeault algebra. In both settings, we equip the categories of modules with model structures of the second kind, whose homotopy categories are, in general, finer invariants than those given by quasi-isomorphisms.

Finally, we introduce a singular analogue of this equivalence, stating it as a zig-zag of Quillen equivalences between the category of dg contramodules over the singular cochain algebra $C^{*}(X)$ and dg presheaves. At the level of homotopy categories, this establishes an equivalence between the contraderived category of  $C^{*}(X)$-contramodules and the homotopy category of dg presheaves.

\end{abstract}

\tableofcontents
\section{Introduction}

In its most basic form, the Riemann-Hilbert correspondence establishes an equivalence between three categories associated to a smooth manifold $X$: the category of flat vector bundles, the category of local systems, and the category of representations of the fundamental group. Several generalizations of this correspondence have been developed to account for the higher homotopy type of $X$ by introducing various notions of \textit{infinity local systems}.

 In particular, Block and Smith \cite{block2014higher} establish an $A_{\infty}$-quasi-equivalence between the category of infinity local systems on $X$ and the category of differential graded (dg) vector bundles on $X$ equipped with a flat $\mathbb{Z}$-graded connection.  They define an infinity local system as an infinity representation of the groupoid $Sing(X)$ of smooth simplices. In other words, a representation of $Sing(X)$ taking values in the category of dg vector spaces, up to homotopy coherence.

A less technical characterization of an infinity local system is as a \textit{cohomologically locally constant (clc) sheaf}—a complex of sheaves whose cohomology sheaves are locally constant. Recently, Chuang, Holstein, and Lazarev \cite{chuang2021maurer} provided a generalization of Block and Smith’s result using the simpler language of clc sheaves. Specifically, they proved that the derived category of perfect clc sheaves is equivalent to the category of \textit{perfect twisted modules} over the de Rham algebra. Perfect twisted modules over the de Rham algebra correspond to dg vector bundles equipped with a flat graded $\mathbb{Z}$-connection.

In the present paper, we enhance both sides of this generalized correspondence by endowing them with model category structures and show that these model categories are equivalent up to a zig-zag of Quillen equivalences. Let $\Omega(X)$ denote the de Rham algebra $\Omega(X)$ on a smooth manifold $X$.  The dg category of perfect twisted modules corresponds to the compact objects in the \textit{compactly generated derived category of the second kind}. This derived category is the homotopy category of a model structure on the category of dg $\Omega(X)$-modules, where weak equivalences are not quasi-isomorphisms but typically a finer invariant.  This model structure provides a natural choice for the model structure on the $\Omega(X)$-module side of the equivalence. 

The global sections functor from sheaves of dg $\Omega$-modules to dg  $\Omega(X)$-modules admits a fully faithful left adjoint: the inverse image functor. It is shown in \cite{chuang2021maurer} this left adjoint quasi-fully faithful when restricted to perfect twisted modules, meaning that its quasi-essential image forms a co-reflective subcategory of the derived category of sheaves of $\Omega$-modules. Since the derived category of sheaves of $\Omega$-modules is equivalent to the derived category of sheaves of dg vector spaces over over the constant sheaf $\ubb$ we can view this co-reflective subcategory as a subcategory of sheaves dg $\ubb$-modules. Motivated by this simple observation we show that this co-reflective subcategory inclusion corresponds to a right Bousfield localisation of model categories. More precisely, we prove the existence of a model category structure on the category of presheaves dg $\ubb$-modules which is equivalent to the model structure of second kind on $\Omega(X)$, up to a zig-zag of Quillen equivalences. 

Moreover, Chuang, Holstein, and Lazarev observed that their construction could be extended beyond smooth manifolds, resulting in a singular analogue of the Riemann-Hilbert correspondence. Specifically, they established an equivalence between the homotopy category of twisted modules over the singular cochain algebra of a topological space and the derived category of clc sheaves. Following their approach, we also extend our results to the singular setting. 

To develop this singular analogue, several subtleties must be addressed. Central to our approach is the use of \textit{contramodules}, which naturally arise from a cohomology theory with coefficients. For instance, the singular cochain algebra on a topological space with coefficients in a vector space inherently possesses the structure of a contramodule. We establish a model structure on the category of presheaves of dg vector spaces the homotopy category of which is model structure is to be equivalent to the derived category of clc sheaves. Our result is then expressed as a zig-zag of Quillen equivalences between the category of contramodules over this singular cochain algebra and the category of presheaves of dg vector spaces. 

Contramodules were originally introduced by Eilenberg and Moore in the 1960s, alongside comodules (see \cite[Chapter III.5]{eilenberg1965foundations}). In contrast to comodules, however, contramodules have received relatively little attention. A contramodule over a coalgebra $C$ is a vector space equipped with a \textit{contraaction map} $Hom(C, V) \lra V$ satisfying unitality and associativity conditions.  Although contramodules were largely forgotten for decades, they resurfaced in the early 2000s thanks to Positselski in his formulation of  \text{Koszul triality}  (see \cite{positselski2011two}). Since then, contramodules have garnered increasing interest, primarily due to Positselski’s work, and a detailed overview of the theory of contramodule theory and their applications can be found in \cite{positselski2011two}.

Before proceeding, we briefly outline our main results. First, when $X$ is a smooth manifold, we equip $\Omega(X)-Mod$ with the model structure of second kind.  Furthermore, there exists a zig-zag of adjunctions between  $\Omega(X)-Mod $ and the dg category of  presheaves of real dg vector spaces, denoted $ PMod(\ubb)$.  We endow the category $ PMod(\ubb)$ with the local model structure, whose homotopy category is the usual derived category of sheaves of dg vector spaces.  Let $L$ denote the set of presheaves in the image of finitely generated twisted modules. Up to stalk-wise quasi-isomorphism, these presheaves are bounded clc sheaves with finite-dimensional fibers. 
\begin{theorem*}
Let $X$ be a smooth connected manifold. Then there exists a zig-zag of Quillen equivalences between $\Omega(X)-Mod $ and the right Bousfield localization of  $ PMod(\ubb)$ at the set $L$.
   
\end{theorem*}
We obtain a similar result in the case where $X$ is a holomorphic manifold. In this situation, we replace $\Omega(X)$ with the $A^{0*}(X)$ Dolbeault algebra of holomorphic forms on X, and we replace $ PMod(\ubb)$ with the dg category of dg presheaves modules over the sheaf of holomorphic functions on $X$, denoted $PMod(\scr{O})$. 

\begin{theorem*}
Let $X$ be a complex manifold. Then there exists a zig-zag of Quillen equivalences between $A^{0*}(X)-Mod $ and the right Bousfield localization of  $ PMod(\scr{O})$ at the set $L$.
\end{theorem*}
Finally, assume that $X$ is a connected, locally contractible topological space and $k$ is a field of characteristic $0$. We denote the dg category of dg contramodules over the pseudo-compact dg algebra of singular cochains on $X$ by $C^{*}(X)-Ctmod$.   We endow $C^{*}(X)-Ctmod$  with Positselski's model structure of second kind the homotopy category of this model category is the contraderived category of $C^{*}(X)$-contramodules, denoted $D^{ct}(C^{*}(X))$.  Let $\mcl{C}$ denote the set of compact generators for the contraderived category $D^{ct}(C^{*}(X))$  and we denote their image in the dg category of presheaves of dg $k$-vector spaces on $X$ by $L$.

\begin{theorem*}
    Let $X$ be a connected and locally contractible topological space and $k$ a field of characteristic $0$. Then there exists a zig-zag of Quillen equivalences between $C^{*}(X)-Ctmod$ and the right Bousfield localization of $PMod(\underline{k})$ at $L$. 
\end{theorem*}

\subsection{Notations and conventions}

We work in the category of $\mathbb{Z}$-graded differential graded (dg) vector spaces over a fixed field $k$ of characteristic $0$. An object in this category is a pair $(V, d_{V})$ where $V$ is a graded k-vector space and $d_{V}$ is a differential, assume all differentials to be of cohomological type. Unmarked tensor products and Homs will be understood to be taken over $k$. We denote the category of dg $k$-vector spaces by $k-Mod$.

A dg algebra is an associative monoid in the category of dg $k$-vector spaces. We work with the category of right dg modules over a dg $k$-algebra $A$, that is, a dg vector space $V$ equipped with a morphism of dg $k$-vector spaces $V \ot A \lra V$ satisfying the usual unitality and associativity conditions. Similarly, a dg coalgebra is a coassociative comonoid in the category of dg $k$-vector spaces and we will work with the category of right dg comodules over a dg coalgebra.  

A dg category in this context is a category enriched over dg $k$-vector spaces. The dg $k$-vector space of morphisms in a dg category $\mcl{T}$ will be denoted $\uhom_{\mcl{T}}(-, -)$ or by $\uhom(-, -)$ when it is clear which ambient category we are working in.  A \textit{dg functor} is a functor enriched over $k-Mod$ and we will refer to $k-Mod$-enriched adjunctions between dg functors as \textit{dg adjunctions}. Moreover, we denote the homotopy category of $\mathcal{M}$ by $H^{0}(\mcl{M})$ that is the category with the same objects as $\mcl{M}$ but with morphisms between two objects $x, y $ given by $\uhom_{H^{0}(\mcl{T})} (m, n) = H^{0}(\uhom_{\mathcal{T}}(x,y))$.

By model category, we mean a complete and cocomplete category with a model structure, as defined in \cite[Definition 1.1.4]{hovey2007model}. For a $\mcl{M}$ be a model category its homotopy category is denoted by $Ho(M)$. If $X$ is an object in $\mcl{M}$, we write $X^{fib}$ and $X^{cof}$ to denote the fibrant and cofibrant replacements of $X$, respectively.

\section{Model Structures and Bousfield localization}

\subsection{Model structures of second kind}

\begin{definition}
    Let $(A, d)$ be a dg algebra. We say that an element $x\in A^{1} $ is \textit{Maurer-Cartan} or \textit{MC} if it satisfies the Maurer-Cartan equation: 

\begin{equation*}
dx+x^2=0.
\end{equation*}
We denote the set of MC elements of $A$ by $MC(A)$.
\end{definition}

\begin{definition}

Let $(A,d_A)$ be a dg algebra and $x\in MC(A)$.
\begin{enumerate}[label=(\roman*)]
\item The \textit{twisting} of $A$ by $x$, denoted $A^x = (A^{x}, d^{x})$, is the dg algebra with the
same underlying graded algebra as $A$ and differential $d^{x}
(a) = d_A(a) + [x, a].$
\item Let $(M, d_{M} )$ be a right dg $A$-module. The \textit{twisting} of $M$ by $x$, denoted $M^{[x]} =
(M, d^{[x]})$, is the right dg $A^{x}$
-module with the same underlying module structure as $M$
and differential $d^{[x]}
(m) = d_{M}(m) + mx$.
\end{enumerate}
\end{definition}

\begin{definition}
A \textit{twisted A-module} over a dg algebra $A$ is a dg $A$-module that is free as an $A$-module after forgetting the differential, that is, it is isomorphic as an $A$-module to a module of the form $V \otimes A $ for some graded
vector space $V$. A \textit{finitely generated twisted} $A$-module is a twisted $A$-module of the form $V \otimes A$ with $V$
finite-dimensional. We say that a twisted module is a \textit{perfect twisted module} if it is a homotopy retract of a finitely generated twisted module. 
\end{definition}

\begin{remark}
\label{twist}
Given any graded vector space $V$, the $A$-module $V \otimes A $ equipped with the differential $1 \ot d_{A}$ is a twisted $A$-module. By considering $V \ot A$ as a $(\underline{End}(V) \ot A)$-$A$-bimodule, any MC element $x \in MC(End(V) \otimes A) $ gives a differential $D= 1 \ot d_{A} + x$ on $V \ot A$ that is compatible with the $A$-module structure. Moreover, any such differential on $V \ot A$ must be of this form (see, \cite[Remark 3.2]{chuang2021maurer}). We denote the dg categories of twisted, finitely generated twisted, and perfect twisted $A$-modules by $TW(A), TW_{fg}(A)$ and $TW_{perf}(A)$, respectively. 

\end{remark}

\begin{theorem} \cite[Theorem 3.7]{guan2021koszul} \label{theorem GL} Let $A$ be an augmented dg algebra. There exists a cofibrantly generated model category structure on the category of dg $A$-modules. Let $f:M \longrightarrow N$ be a morphism in  $A$-Mod. Then, 
\begin{enumerate}[label=(\roman*)]
\item $f$ is a weak equivalence if and only if the induced map 
  \begin{equation*}
  \underline{Hom}((V\otimes A)^{[x]}, M ) \longrightarrow \underline{Hom}((V\otimes A)^{[x]}, N )
  \end{equation*}
  is a quasi-isomorphism  for all finitely generated twisted $A$-modules $(V\otimes A)^{[x]}$, 
  \item $f$ is a fibration if and only if it is  a degree-wise surjection.
\end{enumerate}
We will refer to this model structure as the model structure of the second kind on the category dg $A$-modules. We denote the homotopy category of this model category by $D^{II}(A)$  and we call $D^{II}(A)$ the compactly generated derived category of second kind.  
\end{theorem}
   
\begin{definition}
Let $C$ be a dg coalgebra. We denote the category of dg $C$-comodules by $C-Comod$. Let $X \in C-Comod $ then we say that $X$ is \textit{coacyclic} if it is in the
minimal triangulated subcategory of the homotopy category of dg $C$-comodules containing
the total dg $C$-comodules of exact triples of dg $C$-comodules and closed under infinite direct sums. We say that $X$ is \textit{absolutely acyclic} if it belongs to the minimal thick subcategory of the homotopy category containing the total dg  $C$-comodules of exact triples of dg $C$-comodules.
\end{definition}

\begin{theorem}\label{comod model structure}\cite[Theorem 8.2 (a)]{positselski2011two}
    Let $C$ be a dg coalgebra. There is a cofibrantly generated model category structure on the category of the category of dg $C$-comodules. Let $f: X\lra Y $ be a morphism in $C-Comod$. Then, 

\begin{enumerate}[label=(\roman*)]
  \item f is a weak equivalence if and only if its cone is a coacyclic $C$-comodule, 
  \item f is a fibration if and only if  it is a degree-wise surjection with fibrant kernel,
  \item f is a cofibration if and only if it is injective. 
\end{enumerate}
A dg $C$-comodule is fibrant if and only if it is injective as a graded $C$-comodule. Moreover, the generating set of cofibrations is given by injective maps between finite dimensional comodules and the generating set of acyclic cofibrations is given by injective maps between finite dimensional comodules with absolutely acyclic cokernel and we denote these sets by $I$ and $J$, respectively.  We will refer to the homotopy category of this model structure as the coderived category of dg $C$-comodules and denote it by $D^{co}(C)$.

\end{theorem}
The main objective of the rest of this section is to prove that the model structure of second kind on the category of dg $A$-modules is a $k-Mod$-enriched model category. We begin by briefly recalling the proof of the Theorem \ref{theorem GL}.  Let $DGCOA^{*}$ denote the category of coaugmented coassocitive counital dg $k$-coalgebras and let $DAG^{*}$ denote the of category augmented associative unital dg $k$-algebras by $DGA^{*}$. Then there exists an adjunction  $\Omega: DGCOA^{*} \leftrightarrows DGA^{*}:\widehat{B}$ \cite[Proposition 2.6]{guan2021koszul} . The functor is $\Omega$  the usual cobar construction (see, e.g., \cite[Chapter 2]{loday2012algebraic}). The definition of the functor $\widehat{B}$ can be found in \cite[Definition 2.5]{guan2021koszul}, where it is called the \textit{extended bar} construction. The critical difference between $\widehat{B}$ and the usual bar construction is that $\widehat{B}$ takes an associative unital augmented algebra to a coassociative coaugmented counital coalgebra which is not necessarily conilpotent. The construction of $\widehat{B}$ is given as a functor from $DGA^{*}$ to the opposite category of augmented pseudocompact dg algebras and the latter category is anti-equivalent to the category of coaugmented dg coalgebras via taking duals.   Let $A$ be a dg algebra in $DGA^{*}$, and $\tau: \Omega \widehat{B} A \lra A $ be the \textit{twisting cochain } corresponding to the unit of the adjunction $\Omega \dashv \widehat{B} $. Let $F= - \ot^{\tau} A$  and $G= - \ot^{\tau} \widehat{B}A$, then

\begin{equation*}
    G: \widehat{B}A-Comod \rightleftarrows A-Mod: F
\end{equation*}

is a Quillen equivalence \cite[Theorem 3.10]{guan2021koszul}.   Since the coderived category $D^{co}(\widehat{B}A)$  is compactly generated so too is $D^{II}(A)$.  The compact objects in $D^{II}(A)$ are precisely the perfect twisted $A$-modules and they compactly generate  $D^{II}(A)$  \cite[see, Theorem 4.12.]{guan2023hochschild}.  The model structure on $A-Mod$ is transferred across the adjunction $G \dashv F$ using the following Theorem.

\begin{theorem}[Transfer principle] \cite[Theorem 3.6]{goerss2007model}\label{right transfer}
Let $\mcl{M}$ be a cofibrantly generated model category with $\bar{I}$  being the set of generating cofibrations and $\bar{J}$  being the set of generating acyclic cofibrations. Let $\mcl{C}$ be another category with small colimits and finite limits and suppose that there exists a pair of adjoint functors 

\begin{equation*}
L: \mcl{M} \leftrightarrows \mcl{C}:R.
\end{equation*}
Define a map $f$ in $\mcl{C}$ to be a fibration if $R(f)$ is a fibration in $\mcl{M}$ and a weak equivalence if $R(f)$ is a weak equivalence in $\mcl{M}$. These two classes determine a model structure on $\mcl{C}$ provided the following conditions are satisfied: 

\begin{enumerate}[label=(\roman*)]
\item $R$ preserves compact objects,
\item $\mcl{C}$ has functorial fibrant replacement and functorial path objects for fibrant objects. 
\end{enumerate}
In addition, $\mcl{C}$ will be cofibrantly generated with $L(\bar{I})$ (respectively $L(\bar{J})$) generating cofibrations (respectively acyclic cofibrations) and the adjunction $L \dashv R$ will be Quillen.

\end{theorem}
We will now show that if $\mcl{M}$ is an enriched model category and $L \dashv R$ is an enriched adjunction then $\mcl{C}$ is also model category.  

\begin{definition}
Let $\mathcal{V}$ be a closed symmetric monoidal category and let $\mathcal{C}$ be a $\mathcal{V}$-enriched category, then we say that: 

\begin{enumerate}[label=(\roman*)]
    \item $\mathcal{C}$ is  \textit{cotensored} over $\mathcal{V}$ if there exists a functor $[-, -]_{\mcl{C}}: \mcl{V}^{op} \times \mcl{C} \lra \mcl{C}$ such that for each $v \in \mcl{V}$ and $c_{1}, c_{2} \in \mcl{C}$ there is a natural isomorphism in $\mcl{V}$

    \begin{equation*}
        \mcl{V}(v, \mcl{C}(c_{1}, c_{2})) \simeq \mcl{C}(c_{1},[v, c_{2}]_{\mcl{C}}),
    \end{equation*}
    \item $\mcl{C}$ is \textit{tensored} over $\mathcal{V}$ if there exists a functor $-\ot^{\mcl{C}}- :V \times C \lra C$ such that for each $v \in \mcl{V} $ and $c_{1}, c_{2} \in \mcl{C}$ there is a natural isomorphism in \textit{V}
    \begin{equation*}
        \mcl{C}(v \ot^{\mcl{C}} c_{1}, c_{2}) \simeq \mcl{V}(v,\mcl{C}(c_{1}, c_{2})).
    \end{equation*}
    
\end{enumerate}
We refer to these functors as the \textit{cotensor} and \textit{tensor}, respectively.
\end{definition}
\begin{proposition}
    Let $C$ be a dg coalgebra. Then the category $C$-Comod is tensored and cotensored over $k-Mod$. 
\end{proposition}

\begin{proof}

Let $V$ be a dg vector space and let $X$ be a dg $C$-comodule.  We define the tensor $V \ot X$ of $V$ and $X$ to be the vector space $V \ot_{k} X$ equipped with the following coaction induced by the coaction of $C$ on $X$:
    \begin{align*}
        V \ot X &\lra (V \ot X) \ot C \\
        v \ot x &\longmapsto v \ot \rho (x),
    \end{align*}
where $\rho (x)$ is the coaction of $C$ on $X$. 

We will denote the cotensor by $\underline{hom}(-, -)$. In the case that $V$ is a finite dimensional dg vector space then we define $\underline{hom}(V, X)$ to be the usual homomorphism complex of vector spaces, with coaction defined by the composition 

\[
\begin{tikzcd}
\underline{hom}(V, X)  \arrow[r] & V^{*} \ot X \arrow[r] &  V^{*} \ot X \ot C  \arrow[r] & \underline{hom}(V, X) \ot C 
\end{tikzcd}
\]
where the first and last maps are isomorphisms and the other map is the map induced by the coaction of $C$ on $X$. In the case that $V$ is infinite dimensional define 
\begin{equation*}
    \underline{hom}(V, X) := \clim_{U}   \underline{hom}(U, X) .
    \end{equation*}
Here the limit runs over all the finite dimensional subspaces $U$ of $V$. Let $X$ and $Y$ and be dg $C$-comodules and $V$ a dg vector space. The following calculation shows that the cotensor satisfies the required natural isomorphisms 
\begin{align*}
      \underline{Hom}_{C}(X, \underline{hom}(V, Y))  & \cong \underline{Hom}_{C}(X, \clim_{U}  \underline{hom}(U, Y))\\ &\cong \clim_{U} \underline{Hom}_{C}(X, \underline{hom}(U, Y))\\
                                                  &\cong {\clim_{U}}   \underline{Hom}_{k}(X \ot U, Y)\\
                                                  &\cong  \underline{Hom}_{k}( \colim_{U}  X \ot U, Y )\\
                                                  &\cong \underline{Hom}_{k}(X \ot V, Y)\\
                                                  &\cong \underline{Hom}_{k}(V, \underline{Hom}_{C}(X, Y)).  
\end{align*}

\end{proof}
\begin{definition}
Let $\mathcal{V}$ be a symmetric monoidal model category and $\mathcal{C}$ be a model category. We say that $\mathcal{C}$ is a $\mathcal{V}$-\textit{model category} if $\mathcal{C}$ is a $\mathcal{V}$-enriched category that is tensored and cotensored over $\mcl{V}$ and the following conditions are satisfied: \begin{enumerate}[label=(\roman*)]
    \item For a cofibrant replacement $1_{\mathcal{V}}^{cof} \longrightarrow 1_{\mcl{V}}$ of the unit of $\mcl{V}$, the map 
    \begin{equation*}
        1_{\mcl{V}}^{cof} \ot C \longrightarrow 1_{\mcl{V}} C \cong C
    \end{equation*}
    is a weak equivalence in $\mcl{C}$ for all $C$ in $\cal{C}$
    \item For $i: V \lra V' $ a cofibration in $\cal{V}$ and $j: C \lra C'$ a cofibration in $\mcl{D}$ the pushout product $i \square j$:

    \[
\begin{tikzcd}
V \ot C \arrow[d, "i \ot C"] \arrow[r, "V \ot j"] &
V \ot C' \arrow[d] \arrow[ddr,bend left, "i \ot C'"] \\
V' \ot C \arrow[r] \arrow[drr,bend right, "V' \ot j"] &
V' \ot C \displaystyle{\bigsqcup_{V \ot C }} V \ot C' \arrow[dr,"i \square j"] \\
&& V' \ot C'
\end{tikzcd}
\]
is a cofibration in $\mcl{C}$, moreover, if $i$ or $j$ is a weak equivalence then so is $i \square j$. We call this condition the \textit{pushout product axiom}. 
\end{enumerate}
\end{definition}
In this paper we will only consider the case where  $\mcl{V}$ is the category of differential graded vector spaces, $k-Mod$. We will refer to $k-Mod$-enriched adjunctions as dg adjunctions. Unless otherwise specified, we assume that $k-Mod$ is equipped with the projective model structure. With this model structure, $k-Mod$ is a symmetric monoidal model category (see, e.g., \cite[Proposition 4.2.13]{hovey2007model}). We will refer to $k-Mod$-model categories as dg model categories. 

\begin{remark}\label{derived hom}

Let $\mathcal{M}$ be a dg model category.  The homotopy category of $\mcl{M}$ is naturally enriched over the derived category dg vector spaces $D(k)$. In particular, the enriched Hom-functor $\uhom(-,-): \mcl{M} \times  \mcl{M}  \lra k-Mod$  can be lifted to a right derived functor $ \rdh(-,-):Ho(\mcl{M}) \times Ho(\mcl{M}) \lra D(k)$ defined by 
\begin{equation*}
    \mathbb{R}\uhom(X, Y) : = \uhom(X^{cof}, Y^{fib}), 
\end{equation*}
for all $X, Y \in \mcl{M}$. Here $X^{cof}$ and $Y^{fib}$ denote the cofibrant and fibrant replacement, respectively. As convenient consequence, we can explicitly compute 
\begin{equation*}
     \underline{Hom}_{Ho(\mcl{M})}(X, Y) \cong H^{0}(\rdh(X, Y)).
\end{equation*}
\end{remark}

    \begin{proposition}\label{comod is dg}
With the model structure described in Theorem \ref{comod model structure}, $C$-Comod is a dg model category.
\end{proposition}

\begin{proof}

Since the unit in $k-Mod$ is cofibrant, we only have to check the pushout product axiom. Let $I$ and $J$ be as in Theorem \ref{comod model structure}. Recall the projective model structure on $k-Mod$ is cofibrantly generated with the set of inclusions $I'=\{S^{n-1} \longrightarrow D^{n}\}$ generating cofibrations and $J'= \{0 \longrightarrow D^{n}\}$ generating acyclic cofibrations. Here $S^{n}$ denotes the complex with $k$ in degree $n$ and zero elsewhere. Similarly, $D^{n}$ denotes the complex with $k$ in degree $n$ and $n-1$ and $0$ elsewhere.

By \cite[Corollary 4.2.5.]{hovey2007model} it is sufficient to check the pushout product axiom on the generating sets of cofibrations and acyclic cofibrations. More precisely, it is sufficient to check that $I \square I'$ consists of cofibrations, $J \square I'$ consists of acyclic cofibrations and  $I \square J'$ consists of acyclic cofibrations. Let $i: S^{n-1} \lra D^{n} \in I'$ and $f:X \lra Y  \in I$. The pushout product of $i$ and $f$ is described by the following diagram

\[
\begin{tikzcd}
S^{n-1} \ot X \arrow[d, "i \ot X"] \arrow[r, "S^{n-1} \ot f"] &
S^{n-1} \ot Y \arrow[d, "\beta"] \arrow[ddr,bend left, "i \ot Y"] \\
D^{n} \ot X \arrow[r, "\alpha"] \arrow[drr,bend right, "D^{n} \ot f"] &
P \arrow[dr,"i \square f"] \\
&& D^{n} \ot Y.
\end{tikzcd}
\]
 The maps $S^{n-1} \ot f$ and $ i \ot X$ are injective and, since pushouts preserve injective maps in any abelian category, $\alpha$ and $\beta$ are also injective. It is then straightforward to check that $i\square f$ is injective.

Now assume that $f \in J$, then it can be seen that $\alpha, \ D^{n} \ot f \in J$ by the following reasoning. Observe that the cokernel of $\alpha$ is isomorphic to $\coker(S^{n-1} \ot f) \cong S^{n-1} \ot \coker(f) $, since pushouts preserve cokernels in any abelian category and $\coker(D^{n} \ot f) \cong D^{n} \ot \coker(f)$. It follows by the two out of three axiom that $i\square f $ is a weak equivalence and therefore an acyclic cofibration.

Let $j: 0 \lra D^{n} $ be a map in $J'$ then the pushout product is given by $j \square f = D^{n} \ot f: D^{n} \ot X \lra D^{n} \ot Y$. This map is an injective map between finite dimensional comodules and its cokernel is given by $D^{n} \ot \coker(f)$ which is homotopy equivalent to $0$ and therefore coacyclic.  
\end{proof}

\begin{proposition}\label{enriching prop}\cite[Proposition 3.7.10]{riehl2014categorical}
Let $\mcl{V}$ be a closed symmetric monoidal category and $\mcl{C}$ and $\mcl{D}$ be $\mcl{V}$-enriched categories which are tensored and cotensored over $\mcl{V}$. Suppose that there exists an adjunction $L:\mcl{C} \rightleftarrows \mcl{D}: R$ between the underlying categories. Then the following are equivalent:

    \begin{enumerate}[label=(\roman*)]
    \item The adjunction $L\dashv R$ is a $\mcl{V}$-adjunction.
       \item The functor $R$ is a $\mcl{V}$-functor and there exits a natural isomorphism \\  $R([v, d]_{\mcl{D}}) \cong [v, R(d) ]_{\mcl{C}} $ for all $v \in \mcl{V}$ and $d\in \mcl{D}$.
              \item The functor $L$ is a $\mcl{V}$-functor and there exists a natural isomorphism \\ $L(v \ot^{\mcl{C}} c) \cong v \ot^{\mcl{D}}L(c)$ for all $v \in \mcl{V}$ and $c \in \mcl{C}$.
    \end{enumerate}

\end{proposition}
     
\begin{lemma}\label{Tranferlem}
    Let $\mcl{C}$ be a model category and 

    \begin{equation*}
    L: \mcl{C} \rightleftarrows \mcl{D}: R
\end{equation*}
be an adjunction such that the model structure on $\mcl{C}$ may be transferred to $\mcl{D}$ via right transfer.  Suppose further that $\mcl{C}$ is a $\mcl{V}$-model category for some symmetric monoidal model category $\mcl{V}$, and  $\mcl{D}$ is a $\mcl{V}$-enriched category that is tensored and cotensored over $\mcl{V}$. If the adjunction satisfies any of the equivalent conditions in Proposition \ref{enriching prop}, then $\mcl{D}$ is $\mcl{V}$-model category. 

\end{lemma}

\begin{proof}
Let $[-, -]_{\mcl{C}}$ and $[-, -]_{\mcl{D}}$ denote the cotensoring in $\mcl{C}$ and $\mcl{D}$, respectively. By \cite[Remark A.3.1.6]{lurie2009higher} it is sufficient to check that for cofibration $v_{1} \lra v_{2}$ in $\mcl{V}$ and fibration $d_{1} \lra d_{2}$ in $\mcl{D}$ the induced morphism 
\begin{equation}\label{eq 222}
    [v_{2}, d_{1}]_{\mcl{D}} \lra [v_{1}, d_{1}]_{\mcl{D}} \times_{[v_{1},d_{2}]_{\mcl{D}}} [v_{2},d_{2}]_{\mcl{D}}
\end{equation}
is a fibration and is an acyclic fibration if, in addition, either $v_{1} \lra v_{2}$ or $d_{1} \lra d_{2}$ is a weak equivalence. Note that $R$ is right adjoint, so it preserves pullbacks, and by assumption $R$ commutes with cotensoring. Hence, $(\ref{eq 222})$ is equivalent to requiring that 
\begin{equation*}
    [v_{2}, R(d_{1})]_{\mcl{C}} \lra [v_{1}, R(d_{1})]_{\mcl{C}} \times_{[v_{1},R(d_{2})]_{\mcl{C}}} [v_{2}, R(d_{2})]_{\mcl{C}}
\end{equation*}
is a (acyclic) fibration. The result follows by recalling that $d_{1} \lra d_{2}$ is an (acyclic) fibration if and only if $R(d_{1} \lra d_{2})$ is an (acyclic) fibration. 
\end{proof}

\begin{corollary}
   With the model structure of second kind, $A$-Mod is a dg model category.  
\end{corollary}

\begin{proof}
    This follows from Lemma \ref{Tranferlem}, Proposition \ref{comod is dg} together with noting that the adjunction $G \dashv F$ is a dg adjunction. 
\end{proof}

\subsection{Model Structures on Presheaves}
Let $X$ be a topological space assumed to be  connected, paracompact and Hausdorff. Let $\scr{R}$ be a presheaf of $k$-algebras on $X$ and let $PMod(\scr{R})$ denote the category of presheaves of dg $\scr{R}$-modules.

\begin{theorem}
There exists a cofibrantly generated model structure on the category $PMod(\scr{R})$. Let $f: \scr{F} \lra \scr{G}$ be a morphism of presheaves. Then,

\begin{enumerate}[label=(\roman*)]
\item $f$ is a weak equivalence if and only if $f(U): \scr{F}(U) \lra \scr{G}(U)$ is a quasi isomorphism for all open $U \subset X$,
\item $f$ is a fibration if and only if $f(U): \scr{F}(U) \lra \scr{G}(U)$ is a surjection for all open $U \subset X$,
\item $f$ is a cofibration if  and only if it is a degree-wise split injection with cofibrant cokernel.  
\end{enumerate}

\end{theorem}
\begin{proof}
    This follows from \cite[Theorem 11.6.1]{hirschhorn2003model}.
\end{proof}
We will refer to this model structure as the global projective model structure.

\begin{theorem}\label{Local model Strucutre}\cite[Theorem 5.7.]{choudhury2019homotopy}
There exists a cofibrantly generated model structure on the category of presheaves $PMod(\scr{R})$. Let $f: \scr{F} \lra \scr{G}$ be a morphism of presheaves. Then, 
\begin{enumerate}[label=(\roman*)]
\item $f$ is a weak equivalence if and only if the induced map after sheafification $f^{+}: \scr{F}^{+} \lra \scr{G}^{+}$ is a quasi isomorphism,
\item $f$ is a fibration if and only if $f(U): \scr{F}(U) \lra \scr{G}(U)$ is a surjection for all open $U \subset X$ and the kernel of $f$ is a hypersheaf,
\item is a cofibration if it is a cofibration in the global projective model structure.
\end{enumerate}
 
\end{theorem}
We will refer to this model structure as the \textit{local model structure} and we will often refer to weak equivalences in this model structure as \textit{local weak equivalences}. In order to distinguish which model structure we are working with we will use the notation $PMod^{G}(R)$ for the global projective model structure and $PMod(\scr{R})$ for the local model structure. 
\begin{remark}
 The local model structure can constructed by taking the left Bousfield localisation of the global projective model structure at all "\textit{hypercovers}." The prototypical example of a hypercover is the \u{C}ech  nerve of an ordinary open cover of a topological space. Since we do not make use of hypercovers we will not introduce this terminology. However we will make some simple observations, see \cite{dugger2004hypercovers} for details regarding hypercovers and a simplicial analogue of Theorem \ref{Local model Strucutre}.  

Firstly, since the local model structure is a left Bousfield localisation by \cite[Theorem3.3.4]{hirschhorn2003model} the identity, $Id:PMod^{G}(\scr{R}) \lra PMod(\scr{R})$ is left Quillen.  Secondly, the fibrant objects in the local model structure are precisely the "hypersheaves". We say that a presheaf $\scr{F} $ is a hypersheaf if for any open hypercover $V_{\bullet} \lra U$ of an open subset $U \subset X$ the map 
\begin{equation}\label{decent}
    F(U) \lra  \holim F(V_{\bullet})
\end{equation}
a weak equivalence.  More concretely, (\ref{decent} ) is equivalent to requiring that 

\begin{equation*}
    \scr{F}(U) \lra \check{C}(V_{\bullet}, \scr{F})
\end{equation*}
is a quasi-isomorphism, where $\check{C}(V_{\bullet}, \scr{F})$ denotes the \u{C}ech complex of the hypercover $V_{\bullet} \lra U $ defined analogously to the usual \u{C}ech complex of an ordinary open cover. Since the weak equivalences in the local model structure are maps which are quasi-isomorphisms after sheafification the homotopy category $Ho(PMod(\scr{R}))$ is equivalent to the usual derived category of sheaves of dg $\scr{R}^{+}$-modules.

\end{remark}

Let $\scr{A}$ be a presheaf of non-negatively graded dg $\scr{R}$-algebras and $PMod(\scr{A})$ denote the category of presheaves of $\scr{A}$-modules. Then the map $\scr{R} \lra \scr{A}$ induces an adjunction $- \otimes_{\scr{R}} \scr{A}: PMod(\scr{R}) \rightleftarrows PMod(\scr{A}): J $  where $J$ is the forgetful functor.  The global projective model structure can be transferred via the adjunction $- \otimes_{\scr{R}} \scr{A}: PMod^{G}(\scr{R}) \rightleftarrows PMod(\scr{A}): J $  to the category $PMod(\scr{A})$. Then the category $PMod(\scr{A})$ can be endowed with the local model structure defined in the same way as the local model structure on $PMod(\scr{R})$.  We denote the local model structure and global projective model structure on presheaves of dg $\scr{A}$-modules by $PMod(\scr{A})$ and $ PMod^{G}(\scr{A})$, respectively.

The dg category $PMod(\scr{R})$ is tensored and cotensored over $k-Mod$.  The tensor $V \ot \scr{F}$ is defined to be the presheaf $U \longmapsto V \ot \scr{F}(U)$, where the tensor product $ V \ot \scr{F}(U)$ is the tensor product of $V$ and $F(U)$ as dg vector spaces with action induced by the action of $\scr{R}(U)$ on $\scr{F}(U)$. The cotensor is given by the presheaf $U \longmapsto Hom(V, \scr{F}(U))$, where $Hom(V, \scr{F}(U))$ is defined to be the usual homomorphism complex of the underlying dg vector spaces with action induced by the action of $\scr{R}(U)$ on $\scr{F}(U)$. Analogously, the dg category $PMod(\scr{A})$ is tensored and cotensored over $k-Mod$.

\begin{proposition}
With the tensoring and cotensoring defined as above the model structures are dg model categories:

\begin{enumerate}[label=(\roman*)]
\item $PMod^{G}(\scr{R})$,
\item $PMod(\scr{R})$,
\item $PMod^{G}(\scr{A})$,
\item $PMod(\scr{A})$.

\end{enumerate}
\end{proposition}
\begin{proof}
    See \cite{choudhury2019homotopy}.
\end{proof}

As noted in Remark \ref{derived hom} the local and global projective model structures come equipped with right derived hom functors. We denote the right derived hom functor for the local model structure by $\rdh(-,-)$  and we denote the right derived hom functor in the global projective model structure by $\rdh^{G}(-,-)$.  From from this point onwards we will assume that the morphism $\scr{R} \lra \scr{A}$ is a local weak equivalence. Since $\scr{R} \lra \scr{A}$ is a local weak equivalence, the adjunction $- \otimes_{\scr{R}} \scr{A} \dashv J $ induces an equivalence between the derived category of sheaves of dg $\scr{R}^{+}$-modules and the derived category of sheaves of dg $\scr{A}^{+}$-modules.

\begin{proposition}\label{unlocal qeq}
Let $\scr{R} \lra \scr{A}$ be a stalk-wise quasi-isomorphism. Then the adjunction $- \otimes_{\scr{R}} \scr{A}: PMod(\scr{R}) \rightleftarrows PMod(\scr{A}): J $ is a Quillen equivalence of dg model categories. 
\end{proposition}

\begin{example}
  Let $\ubb$ denote the constant presheaf on a smooth connected manifold $X$. Let $\Omega $ denote the dg sheaf of de Rham algebras on $X$. By the Poincar\'e Lemma $\Omega $ is a soft resolution of the constant sheaf $\ubb^{+}$, in particular a local weak equivalence.  By  Proposition  \ref{unlocal qeq}, and regarding $\Omega$ as a presheaf we can recover the well known equivalence of derived categories $D(\ubb^{+}) \simeq D(\Omega)$. 
\end{example}

\subsection{Localising dg model categories} 

In this subsection we recall some technical machinery relating to Bousfield localisation of model categories which we specialise to the setting of stable dg model categories.  Most notably, we show that under mild assumptions right Bousfield localisation of a stable dg model category is once again a stable dg model category. Let $\mathbf{sSet}$ denote the category of simplicial sets. We endow $\mathbf{sSet}$ with the Kan-Quillen model structure \cite[Section II.3]{quillen2006homotopical}.

\begin{remark}\label{homotopy function complex remark}
For $X, Y \in \mcl{M}$ there exists a simplicial set $Map_{\mcl{M}}(X, Y) $ called the \textit{homotopy function complex} (see, e.g., \cite[Section 5.4]{hovey2007model}). By \cite[Theorem 5.4.9]{hovey2007model}, the homotopy function complex provides an enrichment of $Ho(\mcl{M})$ over the $Ho(\mathbf{sSet})$, in fact, $Ho(\mcl{M})$ is naturally tensored and cotensored over $Ho(\mathbf{sSet}) $, as well as enriched over it. In particular, if $F$ is a left Quillen functor between model categories with right adjoint $U$, we have 

\begin{equation*}
    Map(\mathbb{L}F(X), Y  ) \cong Map (X, \mathbb{R}U(Y))
\end{equation*}
in $Ho(\mathbf{sSet}) $. 
\end{remark}

\begin{definition}
Let $\mathcal{M}$ be a model category and $K$ a set of objects. We say that a morphism $f: A \longrightarrow B $ is a $K$\textit{-coequivalence} if 
\begin{equation*}
    Map_{\mathcal{M}}(X, f) : Map_{\mathcal{M}}(X,A) \longrightarrow Map_{\mathcal{M}}(X,B) 
\end{equation*}
is a weak equivalence in $\mathbf{sSet}$ for all $X\in K$.

We say that an object $Y$ in $\mcl{M}$ is $K$\textit{-colocal}, if  
\begin{equation*}
    Map_{\mathcal{M}}(Z, f) : Map_{\mathcal{M}}(Z,A) \longrightarrow Map_{\mathcal{M}}(Z,B) 
\end{equation*}
is a weak equivalence in $\mathbf{sSet}$, for any $K$-coequivalence $f$.  
\end{definition}

\begin{definition}\label{proper}
We say that a model category $\mathcal{M}$ is \textit{right proper,} if weak equivalences are preserved by pullbacks along fibrations. That is, for each weak equivalence $f: A \longrightarrow B$ and any fibration $h:C \longrightarrow B $, the pullback $\hat{f}: A \times_{C} B \longrightarrow C$ of $f$ is also a weak equivalence. Dually, we say that a model category is \textit{left proper,} if weak equivalences are preserved by pushouts along cofibrations. A model category is said to be \textit{proper} if it is both left and right proper.  

\end{definition}

\begin{example}
Let $\mcl{M}$ be a model category. If every object in $\mcl{M}$ is fibrant then $\mcl{M}$ is a right proper model category  (see \cite[Corollary 13.1.3]{hirschhorn2003model}). In particular, the global projective model structure on presheaves is right proper. In fact, it is known that the projective model structure and local model are proper see, \cite[Corollary 3.1 and Theorem 5.7]{choudhury2019homotopy}. 
\end{example}

\begin{definition}\label{combinatorial def}
We say that a cofibrantly generated model category is \textit{combinatorial} if it is locally presentable as a category.  
\end{definition}

\begin{proposition}\label{existance}
Let $\mcl{M}$ be a combinatorial, right proper model category and $K$ be any set of objects. Then $R_{K}\mcl{M}$ exists and is right proper and combinatorial. Moreover, the identity functor $Id_{\mcl{M}} : \mcl{M} \longrightarrow R_{K} \mcl{M}$ is a right Quillen functor.
\end{proposition}
\begin{proof}
See \cite[Section 5]{barwick2010left}. 
\end{proof}

\begin{remark}
In the existence theorems of \cite[Chapter 4 and 5]{hirschhorn2003model} for Bousfield localisations, Hirschhorn requires that the model category is \textit{cellular} (see, \cite[Definition 12.1.1.]{hirschhorn2003model}) in the place of combinatorial. These are distinctly different properties, for example  the model structure on the category of sets given in \cite[Example 12.1.7]{hirschhorn2003model} is combinatorial but not cellular. On the other hand the category of topological spaces is cellular, see \cite[Proposition 12.1.4. ]{hirschhorn2003model}, but not every topological space is small (see, e.g. \cite[page 49]{hovey2007model}).

Dugger's Theorem \cite[Theorem 1.1]{dugger2001combinatorial} states every combinatorial model category, $\mcl{M}$  is Quillen equivalent to a left Bousfield localisation of the global projective model category on simplicial presheaves over some small category. Since the category of simplicial presheaves is cellular the right Bousfield localisation exists and Bousfield localisation of simplicial presheaves can be lifted back to  $\mcl{M}$. We also note that, by \cite[Proposition 2.3]{dugger2001combinatorial}, combinatorial model categories have functorial fibrant and cofibrant replacement. 

\end{remark}

\begin{definition}
Let $\mcl{M}$ be a pointed model category and $*$ denote the zero object. For $X$ cofibrant, we define the suspension object $\Sigma X$ of $X$ to be the pushout of the following diagram:

\[
\begin{tikzcd}
*   
& X \coprod X \arrow[r, "i"] \arrow[l]
& Cyl(X).
\end{tikzcd}
\]
For $Y$ fibrant, we define the loop object $\Omega Y $ of $Y$ to be the pullback of the following diagram:

\[
\begin{tikzcd}
 P(Y)  &  X \times X \arrow [r, "p"] \arrow[l]  &*.
\end{tikzcd}
\]

\end{definition}

\begin{proposition}\cite[Proposition 3.1.7]{barnes2020foundations}
Let $\mcl{M}$ be a pointed model category. Then the loop and suspension constructions define a pair of adjoint functors $\Sigma : Ho(\mcl{M}) \leftrightarrows Ho(\mcl{M}):\Omega $.
\end{proposition}

\begin{definition}
We say that a pointed model category is \textit{stable} if the adjunction $\Sigma : Ho(\mcl{M}) \leftrightarrows Ho(\mcl{M}): \Omega $ is an equivalence of categories. 
\end{definition}

 Let $\mcl{M}$ be a pointed model category then by Remark \ref{homotopy function complex remark} we have the following weak equivalences of simplicial sets
\begin{align*}
    Map_{M}(\Sigma X, Y ) \cong Map_{M}( X, \Omega Y ).
\end{align*}
Additionally,
\begin{equation*}
 Map_{M}( X, \Omega Y )  \cong \Omega Map_{M}(X,  Y ) 
\end{equation*}

by \cite[Chapter 6 ]{hovey2007model}. Let $K$ a class of objects in $M$. Then $K$ is usually called \textit{stable} if the class of $K$-colocal objects is closed under $\Omega$. However, under the assumption that $M$ is stable this is equivalent to the requirement that the class of $K$-coequivalences is closed under $\Sigma$ or equivalently $K$ is closed under $\Omega$. Since we exclusively work with stable model categories we make the following definition.

\begin{definition}\label{stable def}
  Let $\mcl{M}$ be a stable model category and $K$ a class of objects, then $K$ is called \textit{stable} if it is closed under $\Omega$.   
\end{definition}

\begin{remark}\label{stability}

The local and global projective model structures on $PMod(\scr{A})$ and $PMod(\scr{R})$ are stable and the loop and suspension constructions are simply given by shifting the complex to the left or right by one degree. 

We will check this for the suspension object associated to a cofibrant presheaf in $PMod(\scr{R})$, the other cases are similar. Let $(\scr{F}, d) $ be a cofibrant presheaf. A cylinder object for $\scr{F}$ is given by the complex given with $\scr{F}^{n} \oplus \scr{F}^{n+1} \oplus \scr{F}^{n}$ in degree $n$ and differential represented by the matrix 
\begin{equation*}
\begin{bmatrix}
d^{n} & Id & 0\\
0 & -d^{n+1} & 0\\
0 & -Id& d^{n}
\end{bmatrix}.
\end{equation*}
Then the fold map $\scr{F} \oplus \scr{F} \lra \scr{F}$ factors as the composition $ p \circ i$ where $i$ is given by the matrix 
\begin{equation*}
\begin{bmatrix}
Id & 0 \\
0 & 0\\
0 & Id
\end{bmatrix}
\end{equation*}
and $p$ is given by the matrix $[Id, 0, Id]$. Let $q: \scr{F} \lra Cyl(\scr{F})$ be given by
\begin{equation*}
\begin{bmatrix}
Id \\
0 \\
0 
\end{bmatrix}.
\end{equation*} 
Then $p \circ q = Id $ and $q \circ p $ is homotopy equivalent to the identity on $Cyl(\scr{F})$ via the map 
\begin{equation*}
\begin{bmatrix}
0 & 0 & 0\\
0 & 0 & -Id\\
0 & 0 & 0
\end{bmatrix}.
\end{equation*} 
Hence, $p$ is a weak equivalence and the cokernel of $p$ is given by $\scr{F}$ shifted to the right by one degree, which is cofibrant by assumption. Furthermore, the suspension object $\Sigma\scr{F}$ is defined to be the pushout of $i$ and the zero map which is precisely the cokernel of $i$, hence we recover the usual shift functor for chain complexes. 
\end{remark}

Let $\cal{M}$ be a pointed model category and $X,Y$ be objects in $\cal{M}$. Let $[X, Y]$ denote the set $Hom_{Ho(\cal{M})}(X,Y)$ of morphisms from $X$ to $Y$ in the homotopy category and

\begin{equation*}
    [X,Y]_{n} := \begin{cases}
    [\Sigma^{n}X , Y ] \ \ if \  n\geq 0 \\
    [X, \Sigma^{-n}Y ] \ \ if \ n<0.
    \end{cases}
\end{equation*}

\begin{proposition} \label{stab}
Let $\mcl{M}$ be a pointed model category.

\begin{enumerate}[label=(\roman*)]
    \item The sets $[\Sigma^{n} X, Y ] $ and $[X, \Omega^{n} Y]$ have a group structure for $n \geq 1$ and an abelian group structure for $n \geq 2$. Moreover,  the adjunction isomorphism is an isomorphism of groups (\cite[Lemma 3.2.8]{barnes2020foundations}). 
             \item $\pi_{n} (Map_{\mcl{M}}(X,Y)) \cong [X, Y]_{n} \ for \ n \geq 0 $ (\cite[Lemma 6.9.19]{barnes2020foundations}).
                    \item If $\mcl{M}$ is stable then $Ho(\mcl{M})$ is an additive category. In fact, $Ho(\mcl{M})$ is triangulated (\cite[Proposition 3.2.9, Theorem 4.2.1]{barnes2020foundations}). 
                               \item Suppose $L : \mcl{M} \rightleftarrows \mcl{N} : R$ is a Quillen adjunction between stable model categories then the functors $\mathbb{L}: Ho(\mcl{M}) \lra Ho(\mcl{N})$ and $\mathbb{R}:Ho(\mcl{N}) \lra Ho(\mcl{M}) $ are exact as functors between triangulated categories (\cite[Theorem 4.5.2]{barnes2014stable}).
     
\end{enumerate}

\end{proposition}

\begin{lemma}\label{weak eq loc}
    Let $\mathcal{M}$ be a stable dg model category and let  $K$ be a stable set of objects in $\mcl{M}$. Then the map $f: X \longrightarrow Y$ is a $K$-coequivalence if and only if the induced map $\mathbb{R}\underline{Hom}_{\mcl{M}}(L, f): \mathbb{R}\underline{Hom}_{\mcl{M}}(L, X) \longrightarrow \mathbb{R}\underline{Hom}_{\mcl{M}}(L, Y)$ is a quasi-isomorphism for all $L \in K$. 
\end{lemma}

\begin{proof}
 By Proposition \ref{stab}, $f$ is a $K$-coequivalence if and only if 

\begin{equation*}
    [L,f]_{n}: [L, X]_{n} \longrightarrow [L , Y]_{n}
\end{equation*}
is an isomorphism of abelian groups for all non-negative $n$ and all $L$ in $K$. Now since $\mcl{M}$ is dg model category we have that  $H^{0}(\mathbb{R}\underline{Hom}_{\mcl{M}}(L, X )) \cong [L,X]$. Since $K$ is closed under taking loops and suspensions, we can calculate 

\begin{equation*}
    H^{n}(\mathbb{R}\underline{Hom}_{\mcl{M}}(L, X ))\cong \begin{cases} 
      [\Sigma^{n} L, X], \ \ n\geq 0 \\
        [\Omega^{-n} L,X], \ \ n <  0.
   \end{cases}
\end{equation*}
Hence, $f$ is a $K$-coequivalence if and only if $\mathbb{R}\underline{Hom}_{\mcl{M}}(L, f): \mathbb{R}\underline{Hom}_{\mcl{M}}(L, X) \longrightarrow \mathbb{R}\underline{Hom}_{\mcl{M}}(L, Y)$ is a quasi-isomorphism for all $L \in K$.

\end{proof}

Recall that  the projective model structure on $k-Mod$ is cofibrantly generated with the set of inclusions $I'=\{S^{n-1} \longrightarrow D^{n}\}$ generating cofibrations and $J'= \{0 \longrightarrow D^{n}\}$ generating acyclic cofibrations. Here $S^{n}$ denotes the complex with $k$ in degree $n$ and zero elsewhere. Similarly, $D^{n}$ denotes the complex with $k$ in degree $n$ and $n-1$ and $0$ elsewhere.

\begin{proposition}\label{enriched localisation}
     Let $\mathcal{M}$ be a combinatorial, right proper, stable dg model category and let  $K$ a set of objects in $\mcl{M}$ satisfying the following properties: 
\begin{enumerate}[label=(\roman*)]
\item$K$ is stable,
             \item $K$ is a set of cofibrant objects in $\mcl{M}$.
\end{enumerate}
Then, the right Bousfield  localisation $R_{K}\mathcal{M}$ a stable dg model category. 
\end{proposition}
  \begin{proof}
 Firstly, note that $R_{K}\mcl{M}$ is stable by \cite[Proposition 4.6]{barnes2014stable}. To show that  $R_{K}\mcl{M}$ is a dg model category, we check the push out product axiom. Let $i: V \lra W$ be a morphism in $k-Mod$ and  $f:  X \lra Y $ be a morphism in $\mcl{M}$ and let $[-, -]$ denote the cotensor of $\mcl{M}$ over $k-Mod$.  We can form the pullback diagram: 
      
      \[
\begin{tikzcd}
\left[W, X \right] \arrow[drr, bend left] \arrow[ddr, bend right] \arrow[dr, "\alpha"] \\
&P \arrow[r] \arrow[d] & \left[ W, Y \right] \arrow[d]\\
& \left[V, X\right] \arrow[r] & \left[V, Y \right].
\end{tikzcd}
\]
The pushout product axiom is equivalent to checking the following conditions, by \cite[A.3.1.6]{lurie2009higher}: 
\begin{enumerate}[label=(\roman*)]
\item if $i$ is a cofibration and $f$  is a fibration     then              $\alpha$ is a fibration.
             \item If $i$  is an acyclic cofibration and $f$ is a fibration, then  $\alpha$ is an acyclic fibration in $R_{K}\mathcal{M}$.
             \item If $f$ is an acyclic fibration then $\alpha$ is an acyclic fibration in $R_{K}\mathcal{M}$. 
\end{enumerate} 
The first two follow immediately from the definition of right Bousfield localisation and for (iii) we only need to check that $\alpha$ is a weak equivalence in $R_{K}\mathcal{M}$.  Let $f$ be an acyclic fibration, then by \cite[Proposition 4.3.1]{hovey2007model} it is sufficient to assume that $i \in I'$. To see that $\alpha$ is a weak equivalence in  $R_{K}\mathcal{M}$ we show that $\widehat{f}: P \lra [D^{n}, Y]$ and $[D^{n}, f]: [D^{n}, X] \lra [D^{n}, Y]$  are weak equivalences and appeal to  the  2-out-of-3 axiom.   
 
For any $V$ in $k-Mod$ the functor  $[V,-]: \mathcal{M} \lra \mathcal{M} $ is right Quillen and since every acyclic fibration in $\mathcal{M}$ is an acyclic fibration in $R_{K}\mathcal{M}$ and fibrations are unchanged,  $[V,-]: \mathcal{M} \lra R_{K}\mathcal{M} $ is also right Quillen.

Note that $[D^{n}, f]$ is a weak equivalence if and only if the induced map

\begin{equation}\label{123}
    \rdh_{M}(k, [D^{n}, X] ) \lra \rdh_{M}(k, [D^{n},Y] ) 
\end{equation}
 is a quasi-isomorphism for all $k \in K $.  Let  $[D^{n}, X]^{fib}$  denote the fibrant replacement of $[D^{n}, X]$, then (\ref{123}) is a quasi-isomorphism if and only if 
\begin{equation*}
    \uhom_{M}(k, [D^{n}, X]^{fib} ) \lra \uhom_{M}(k, [D^{n},Y]^{fib} ) 
\end{equation*}
is quasi-isomorphism. This is equivalent to requiring that 

\begin{equation}\label{1234}
    \uhom_{M}(k, [D^{n}, X^{fib}] ) \lra \uhom_{M}(k, [D^{n},Y^{fib}] ) 
\end{equation}
is a quasi-isomorphism.  Since Both $D^{n}$ and $k$ are cofibrant and $\mcl{M}$ is a dg model category (\ref{1234}) is a quasi-isomorphism if and only if 
 \begin{equation}\label{12345}
    \rdh_{M}( D^{n} \ot k ,  X ) \lra \rdh_{M}(D^{n} \ot k,Y).
\end{equation}

is a quasi-isomorphism. 
\begin{equation}\label{123456}
    \rdh_{M}( D^{n} \ot k ,  X ) \lra \rdh_{M}(D^{n} \ot k,Y).
\end{equation}
It follows by the derived tensor cotensor adjunctions for $\mcl{M}$ that (\ref{123456}) is a quasi-isomorphism if and only if 

\begin{equation*}
    \rdh_{k}( D^{n} , \rdh_{\mcl{M}}(k, X) ) \lra \rdh_{k}(D^{n},\rdh_{\mcl{M}}(k, Y)) 
\end{equation*}
is a quasi-isomorphism. An analogous argument shows that $[S^{n}, f]$ is a weak equivalence and therefore an acyclic fibration. Since the pullback of an acyclic fibration is an acyclic fibration, $\widehat{f}$ is an acyclic fibration.

\end{proof}


\begin{remark}\label{triangle stuff}
We can make several observations regarding the Bousfield localization of stable model categories and the classical notion of Verdier quotients in triangulated categories. Let $\mathcal{M}$ be a stable model category and suppose the right Bousfield localisation of $\mathcal{M}$ by a set of objects $K$ exists and is also stable. 

By Proposition \ref{existance}, the identity functor on $\mcl{M}$ induces an adjunction between the homotopy categories which we denote by $\mathbb{L}: Ho(\mathcal{M}) \rightleftarrows Ho(R_{K}\mathcal{M}): \mathbb{R}  $. Moreover, by Proposition \ref{stab} (iii), the homotopy categories are triangulated and by Proposition \ref{stab}(iv),  Quillen functors are exact, thus $L= \ker (\mathbb{R}) $ is a thick subcategory. Then $\mathbb{R}I$ induces a map from Verdier Quotient $J: Ho(\mathcal{M})/L \lra Ho(R_{K}\mathcal{M})$. 
\

To see that this map is an equivalence of categories,  we will show that $R_{K}\mathcal{M} = \mathcal{M} \lra Ho(\mathcal{M} ) \lra Ho(\mathcal{M})/L $ satisfies the universal property  of $R_{K}\mathcal{M} \lra  Ho(R_{K}\mathcal{M}) $. Indeed, suppose $H: \mathcal{M} \lra \mathcal{N}$ is a functor which sends all weak equivalences in $R_{K}\mathcal{M}$ to isomorphisms. Then, since each weak equivalence in $R_{K}\mathcal{M}$ is a weak equivalence in $\mathcal{M}$, the functor $H$ factors through $\mathcal{M} \lra Ho(\mathcal{M}) $ in a unique way. Let $\bar{H} : Ho(\mathcal{M}) \lra N $ denote the functor induced by this factorisation. The Verdier quotient inverts all maps whose cone belong to $L$, and such a map may be represented as a zig-zag of weak equivalences in $R_{K}\mathcal{M}$. Now, since $H$ sends weak equivalences in $R_{K}\mathcal{M} $ to isomorphisms we must have that $\bar{H}$ factors through $Ho(\mathcal{M}) \lra Ho(\mathcal{M})/L$ in a unique way. Collecting this all together, we see that $H$ factors though $R_{K}\mathcal{M} \lra Ho(\mathcal{M})/L$. The uniqueness of this factorisation is ensured by the uniqueness of the previous two factorisations. 

Let $\widehat{l} = \mathbb{L} \circ \mathbb{R} : Ho(\mcl{M}) \lra Ho(\mcl{M}) $  and let $\epsilon $ denote the derived counit. It is easy to see that the pair $(\widehat{l}, \epsilon )$ define an exact colocalisation of triangulated categories.  By \cite[Proposition 4.12.1. ]{krause2008localization} this we can conclude that there exists an exact localisation of triangulated categories $l = \mathbb{L} \circ \mathbb{R} : Ho(\mcl{M}) \lra Ho(\mcl{M}) $ such that $Ho(\mathcal{M})/L \cong im(\widehat{l}) \cong \ker(l)$

\end{remark}
\begin{theorem}\label{triqot}
Let $\mcl{M}$ be a combinatorial, right proper, stable dg model category and $K$ be a stable set of cofibrant objects in $\mcl{M}$. Then there is an equivalence of triangulated categories between $Ho(R_{K}\mcl{M})$ and the Verdier Quotient $Ho(\mcl{M})/L$, where $L = K^{\perp} \subset Ho(\mcl{M})$ is the full subcategory of objects $l $ in $Ho(\mcl{M})$ such that $Hom_{Ho(\mcl{M})}(k, l) = 0 $ for all $K \in L$.

\end{theorem}
\begin{proof}
Firstly, note that $R_{K}\mcl{M}$ is stable by \cite[Theorem 5.9]{barnes2014stable}. It is enough to check that an object $l$ is in $ker(\mathbb{R}I) $ if and only if $Hom_{D(\scr{A})}(\scr{K}, \scr{L} ) = 0$ for all $k \in K $. An object $l$ becomes trivial after applying $\mathbb{R}I$ if and only if $\rdh(k, \mathbb{R}I l ) = \rdh(k, l) $ is quasi-isomorphic to zero for all $k \in K$. Note that 

\begin{equation*}
    H^{n} ((\rdh(k, l)) \cong H^{0} (\rdh(\Sigma^{n}k,  l ))
\end{equation*}
and since $K$ is closed under suspension we have that, 
\begin{equation*}
    H^{n} (\rdh(k, l)) \cong 0  
\end{equation*}
if and only if 
\begin{equation*}
    H^{0} (\rdh(k, l)) \cong 0  
\end{equation*}
for all $k \in K $. The result follows by recalling that 

\begin{equation*}
    H^{0} (\rdh(k, l)) \cong Hom_{Ho(\mcl{M})}(k, l)
\end{equation*}

\end{proof}

\section{Applications to Modules and Presheaves}

\subsection{Localising dg presheaves}
Let $X$ be a paracompact, Hausdorff topological space and let $*$ denote the one point topological space.  As in the previous sections we assume that $\scr{R}$ is a presheaf of dg $k$-algebras on $X$ and $\scr{A}$ is a presheaf of non-negatively graded dg $\scr{R}$-algebras. We assume there exists a stalk-wise quasi-isomorphism $\scr{R} \lra \scr{A}$.   Let $A$ denote the dg algebra $\scr{A}(X)$ and $p$ denote the map $p: (\scr{A}, X) \lra (A, *)$ of dg ringed spaces. Then the usual inverse image and direct image functors form a dg adjunction $p^{*}: A-Mod \leftrightarrows PMod(\scr{A}): p_{*}$. In this case $p_{*}$ is simply the global sections functor and $p^{*}$ takes a dg $A$-module $M$ to the presheaf $U \mapsto M \ot_{A} \scr{A}(U)$.   The category of dg $A$-modules is assumed to be endowed with the model structure of second kind. 

\begin{proposition}\label{og}
  The adjunction  $p^{*}: A-Mod \leftrightarrows PMod^{G}(\scr{A}): p_{*}$ is Quillen.
\end{proposition}
\begin{proof}
It is immediate that $p_{*}$ preserves fibrations. Assume that $f: \scr{F} \lra \scr{G}$ is an acyclic fibration.   Since $T$ is finitely generated,  $p^{*}(T)$ is a bounded below complex of the form $ \underline{V} \ot_{\underline{k}}\scr{A}$ for some finite dimensional vector space V, hence $p^{*}(T)$ is cofibrant (see  \cite[Section 3.2]{choudhury2019homotopy}). In particular,  $\rdh^{G}(p^{*}(T), f)$ is a quasi-isomorphism and since all preshevaes are fibrant, we have that $\uhom(p^{*}(T), f) $ is a quasi-isomorphism for all $T \in TW_{fg}(A)$ and therefore $p^{*}(f)$ is a weak equivalence in $A-Mod$. 
    
\end{proof}
 Let $ K= \{ p^{*}(T) | T\in TW_{fg}(A)\}$ and note that this is a stable set of cofibrant objects in the local and global projective model structure on presheaves of dg $\scr{A}$-modules, hence by Proposition \ref{enriched localisation} we obtain the following results. 

\begin{proposition}\label{localisation of global}
  The right Bousfield localisation of  $ PMod^{G}(\scr{A})$ at $K$ exists and $R_{K}PMod^{G}(\scr{A})$ is a right proper, combinatorial, stable, dg model category.  Let  $f: \scr{F} \lra \scr{G}$ be a morphism of presheaves. Then, 
    \begin{enumerate}[label=(\roman*)]
    \item $f$ is a weak equivalence if and only if the induced map $\uhom_{\scr{A}}(\scr{K}, f)$ is a quasi-isomorphism for all $\scr{K} \in K$, 
    \item $f$ is a fibration if and only if $f(U): \scr{F}(U) \lra \scr{G}(U)$ is surjective for each open set $U \subset X$.
    \end{enumerate}
\end{proposition}

\begin{proposition}\label{localisation of local}
   The right Bousfield localisation of  $ PMod(\scr{A})$ at $K$ exists and  $R_{K} PMod(\scr{A})$ is a right proper, combinatorial, stable, dg model category. Let $f: \scr{F} \lra \scr{G}$: a morphism of presheaves. Then,
    \begin{enumerate}[label=(\roman*)]
    \item $f$ is a weak equivalence if and only if the induced map $\rdh_{\scr{A}}(\scr{K}, f)$ is a quasi-isomorphism for all $\scr{K} \in K$, 
    \item $f$ is a fibration if and only if $f(U): \scr{F}(U) \lra \scr{G}(U)$ is surjective for each open set $U \subset X$ and the kernel of $f$ is a hypersheaf.
    \end{enumerate}
\end{proposition} 

\begin{proposition}\label{localisation of proj}
Let $L = \{ J ( (p^{*}(T))^{fib})^{cof}| T \in TW_{fg}(A)\}$. Then the right Bousfield localisation of  $ PMod(\scr{R})$ at $L$ exists and  $R_{L} PMod(\scr{R})$ is a right proper, combinatorial, stable, dg model category. Let  $f: \scr{F} \lra \scr{G}$ be a morphism of presheaves. Then, 
    \begin{enumerate}[label=(\roman*)]
    \item $f$ is a weak equivalence if and only if the induced map $\rdh_{\scr{R}}(\scr{L}, f)$ is a quasi-isomorphism for all $\scr{L} \in L$,
    \item $f$ is a fibration if and only if $f(U): \scr{F}(U) \lra \scr{G}(U)$ is surjective for each open set $U \subset X$ and kernel of $f$ is a hypersheaf. 
\end{enumerate}
\end{proposition}

 \begin{proof}
     Analogously to Proposition \ref{localisation of global} and Proposition \ref{localisation of local},this result follows from Proposition \ref{enriched localisation}. 
 \end{proof}

\begin{proposition}\label{proposition gen eq}
    The adjunction $- \otimes_{\scr{R}} \scr{A}:R_{L}PMod(\scr{R}) \leftrightarrows R_{K}PMod(\scr{A}): J  $ is a Quillen equivalence of dg model categories. 
\end{proposition}
\begin{proof}
It is clear that $J$ preserves fibrations. By \cite[Corollary A.2 ]{dugger2001replacing}, it is sufficient to check that $- \otimes_{\scr{R}} \scr{A}$ preserves cofibrations between cofibrant presheaves. In order to avoid confusion we refer to cofibrations in $ R_{L}PMod(\scr{R}) $ as  $L$-colocal cofibrations and cofibrations in $PMod(\scr{R}) $ as cofibrations.  In particular, we must show that $- \otimes_{\scr{R}} \scr{A}$ preserves $L$-colocal cofibrations between cofibrant $L$-coclocal presheaves. However, by \cite[Proposition 3.3.16 ]{hirschhorn2003model} $L$-colocal cofibrations between cofibrant $L$-coclocal presheaves coincide with cofibrations in $PMod(\scr{R}) $ which are preserved by $- \otimes_{\scr{R}} \scr{A}$.

Let $\scr{H}$ be a cofibrant $L$-colocal presheaf, then the derived unit $\scr{H} \lra \mathbb{R}J(\scr{H} \ot_{\scr{R}} \scr{A})$ is a local weak equivalence. In particular it is a weak equivalence in $R_{L}PMod(\scr{R})$, since the adjunction

\begin{equation}\label{prop 4.5 eq}
    - \otimes_{\scr{R}} \scr{A}:PMod(\scr{R}) \leftrightarrows PMod(\scr{A}): J  
\end{equation}

is a Quillen equivalence.

By \cite[Corollary 1.3.16. ]{hovey2007model}, to conclude the proof it is enough to check that $J$ reflects weak equivalences between fibrant presheaves. Suppose that $f: \scr{F} \lra \scr{G}$ is a morphism between fibrant presheaves such that $J(f)$ is a weak equivalence. Then for all $\scr{K} \in K$ the induced map 

\begin{equation*}
    \rdh_{\scr{R}}(J(\scr{K}^{fib})^{cof}, J(\scr{F})) \lra  \rdh_{\scr{R}}(J((\scr{K})^{fib})^{cof}, J(\scr{G}))
\end{equation*}
is a quasi-isomorphism. Let $\mathbb{L}$ denote the left derived functor of   $- \otimes_{\scr{R}} \scr{A}:PMod(\scr{R}) \lra PMod(\scr{A})$. Since $\scr{F}$  and $\scr{G}$ are fibrant, passing across the adjunction we obtain a quasi-isomorphism

\begin{equation*}
    \rdh_{\scr{A}}(\mathbb{L}(\mathbb{R}J(\scr{K}), \scr{F})) \lra  \rdh_{\scr{A}}(\mathbb{L} (\mathbb{R}J(\scr{K})), \scr{G}),
\end{equation*}
which is a quasi-isomorphism since $\scr{F}$  and $\scr{G}$ are fibrant. The result follows from the fact that the derived counit of the adjunction (\ref{prop 4.5 eq}) is a local weak equivalence.

\end{proof}

\begin{remark}\label{Godement remark}
    For each $\scr{F}$  in $PMod(\scr{A})$ we can functorially assign a coaugmented cosimplicial presheaf $G^{\bullet}(\scr{F})$ which is defined by taking the Godement resolution  of $\scr{F}$ (see, e.g. \cite[Section 2.1]{demailly1997complex}). More explicitly, let $X_{disc}$ denote the underlying set of $X$ with the discrete topology and let $a: X_{disc} \lra X $ denote the  the continuous map induced by identity. The cosimplicial presheaf $ G^{\bullet}(\scr{F})$ is defined by setting $G^{n}(\scr{F})= (a_{*}a^{*})^{n}(\scr{F})$, where $a_{*}$ and $a^{*}$ denote the inverse direct image functors induced by $a$. Let $G(\scr{F})$ denote the homotopy limit $\mathbb{R}Lim_{\triangle} (G^{\bullet}(\scr{F}))$.  It is shown in \cite[Theorem 6.3]{choudhury2019homotopy} that $G(\scr{F})$ is a hypersheaf and consequently fibrant in the local model structure.

If the category of open sets on $X$ forms a "site of finite type" \cite[Definition 1.31]{morel19991} it is shown  that the morphism $\scr{F} \lra G(\scr{F})$ is a local weak equivalence \cite[Theorem 6.3]{choudhury2019homotopy}. In particular,  the functor $G$ defines a functorial fibrant replacement in the local model structure.  If $X$ has finite topological dimension then this finiteness condition is satisfied for instance if $X$ is a topological manifold of finite dimension (\cite[Section 5.2]{rodriguez2015godement}).  
\end{remark}

\begin{lemma}\label{silly lem}
 Let $\scr{Q}$ be a cofibrant presheaf in $PMod^{G}(\scr{A})$. Then the morphism $\scr{F} \lra G(\scr{F})$ induces a quasi-isomorphism 
 
 \begin{equation*}
     \rdh^{G}(\scr{A}, \scr{F}) \lra \rdh^{G}(\scr{Q}, G(\scr{F}))
 \end{equation*}

\end{lemma}
\begin{proof}
  
    Since $G(\scr{F}) = \mathbb{R}Lim_{\triangle} (G^{\bullet}(\scr{F}))$ is a homotopy limit,  it commutes with the right derived hom functor 
    
        \begin{equation*}
        \rdh_{\scr{A}}^{G}(\scr{Q}, G(\scr{F})) \cong \mathbb{R}Lim_{\triangle}( \rdh_{\scr{A}}^{G}(\scr{Q}, G^{n}(\scr{F})). 
    \end{equation*}

 Let $a_{*}$ and $a^{*}$ be as in Remark\ref{Godement remark}. Note that $a_{*}a^{*}$ commutes with stalk-wise quasi-isomorphisms \cite[Theorem 6.3]{choudhury2019homotopy}. Hence, the canonical morphism $\scr{F} \lra G^{n}(\scr{F})$ induces a  quasi-isomorphism

\begin{equation*}
    \rdh_{\scr{A}}^{G}(\scr{Q}, \scr{F}) \lra \rdh_{\scr{A}}^{G}(\scr{Q}, G^{n}(\scr{F})),
\end{equation*}
for all $n$ and any cofibrant presheaf $\scr{Q}$. Consequently,  after taking the homotopy limit of the right hand side, the morphism
 
\begin{equation*}
    \rdh_{\scr{A}}^{G}(\scr{Q}, \scr{F}) \lra \rdh_{\scr{A}}^{G}(\scr{Q}, G(\scr{F}))
\end{equation*} 

is a quasi-isomorphism. 
\end{proof}

\begin{proposition}\label{silly prop}
    Let $Q$ be a set of presheaves in $PMod^{G}(\scr{A})$ satisfying the conditions of  Proposition \ref{enriched localisation}. Then $ R_{Q}PMod^{G}(\scr{A})$  and $R_{Q}PMod(\scr{A})$  are Quillen equivalent dg model categories. 
\end{proposition}

\begin{proof}
  
    Let $L:PMod^{G}(\scr{A}) \lra  PMod(\scr{A}) $ denote the left Quillen identity functor which and let $R: PMod(\scr{A})  \lra PMod^{G}(\scr{A})$ denote the right adjoint. It is immediate that 
    
    \begin{equation*}
        R: R_{Q}PMod(\scr{A}) \lra R_{Q}PMod^{G}(\scr{A})
    \end{equation*}

preserves fibrations. Let $\mathbb{R}$ denote the right derived functor of $R$. Let $\scr{Q} \in Q$ and $f : \scr{F} \lra \scr{G}$ be a morphism of presheaves. Since $\scr{Q}$ is cofibrant,

\begin{equation*}
    \rdh^{G}(\scr{Q}, \scr{F}) \lra \rdh^{G}(\scr{Q}, \scr{G})
\end{equation*}

is a quasi-isomorphism if and only if

\begin{equation*}
    \rdh^{G}(\scr{Q}, G(\scr{F})) \lra \rdh^{G}(\scr{Q}, G( \scr{G}))
\end{equation*}

is a quasi-isomorphism. By Lemma \ref{silly lem}, $ \scr{F}$ is weakly equivalent to $G(\scr{F})$ in $R_{Q}PMod^{G}(\scr{A})$. In particular,

\begin{equation*}
    \rdh^{G}(\scr{Q}, \scr{F}) \lra \rdh^{G}(\scr{Q}, \scr{G})
\end{equation*}

is a quasi-isomorphism if and only if 

\begin{equation*}
    \rdh(\scr{Q}, \scr{F}) \lra \rdh(\scr{Q}, \scr{G})
\end{equation*}

is a quasi-isomorphism. Therefore $$R: R_{Q}PMod(\scr{R}) \lra R_{Q}PMod(\scr{R}) $$ preserves and reflects weak equivalences.  Let $\scr{F}$ be a cofibrant presheaf then unit of the adjunction $F \lra R(L(\scr{F}))$  is the identity. It follows from   \cite[Lemma 3.3]{erdal2019model}, that the adjunction is a Quillen equivalence. 
\end{proof} 

    \begin{proposition}\label{quillen eq 1}
  The adjunction  $p^{*}: A-Mod \leftrightarrows R_{K}PMod(\scr{A}): p_{*}$ is a Quillen equivalence. 
\end{proposition}
\begin{proof}
    Let $ f: \scr{F} \lra \scr{G} $ be a morphism of presheaves and $T \in TW_{fg}(A)$ then  $f$ is a weak equivalence in $R_{K}PMod(\scr{A})$ if and only if $p_{*}(f)$ is a weak equivalence in $A-Mod$. Clearly, $p_{*}$ preserves fibrations and $p_{*}(p^{*}(M)) \cong M $ thus by \cite[Lemma 3.3]{erdal2019model}  the adjunction is a Quillen equivalence. 
\end{proof}

    \begin{theorem}\label{manin theorem final}
  Let $X$ be a connected, paracompact, Hausdorff topological space and suppose that there exists is  a stalk-wise quasi-isomorphism $\scr{R} \lra \scr{A}.$ Then the adjunctions

   \[
\begin{tikzcd}[row sep=2em]
A-Mod \arrow[r, shift left, "p^{*}"]  &  R_{K} PMod(\scr{A}) \arrow[l, shift left, "p_{*}" ]  \arrow[r, shift right, swap, "J" ] & R_{L}PMod(\scr{R}) \arrow[l, shift right,swap, "-\ot_{\scr{R}} \scr{A}" ]
\end{tikzcd}
\]

  form a zig-zag of Quillen equivalences of dg model categories.   
    \end{theorem}

    \begin{proof}

We have  the following sequence of adjunctions:

     \[
\begin{tikzcd}[row sep=2em]
A-Mod \arrow[r, shift left, "p^{*}"]  &  R_{K} PMod(\scr{A}) \arrow[l, shift left, "p_{*}" ] \arrow[r, shift left, "L"] &  R_{K} PMod(\scr{A}) \arrow[l, shift left, "R"] \arrow[r, shift right, swap, "J" ] & R_{L}PMod(\scr{R}) \arrow[l, shift right,swap, "-\ot_{\scr{R}} \scr{A}" ]
\end{tikzcd}
\]

The first is a Quillen equivalence by Proposition \ref{quillen eq 1};  the second is a Quillen equivalence by Proposition \ref{silly prop} and the final is a Quillen equivalence by Propositon \ref{proposition gen eq}.

    \end{proof}

In the subsequent sections we apply this result to two specific cases; namely the de Rham algebra and the Dolbeault algebra.  Let us first recall some definitions and results regarding dg sheaves from \cite{chuang2021maurer}. Let $\scr{R}$ be a sheaf of dg $k$-algebras on $X$ and  $\scr{A}$ be a sheaf of dg $\scr{R}$-algebras. Let $\scr{R}- Mod$   denote the category of sheaves of dg $\scr{R}$-modules and  $Mod(\scr{A})$ sheaves of dg $\scr{A}$-modules and let $A= \scr{A}(X)$.  The morphism of ringed spaces $q:(\scr{A}, X) \lra (*, A) $ induces an adjunction 
\begin{equation}\label{sheaf inverse}
    q^{*}: A-Mod \leftrightarrows Mod(\scr{A}): q_{*},
\end{equation}
where $q^{*}$ is the inverse image functor and $q_{*}$ is the direct image functor. We also have an adjunction between the restriction and extension of scalars functors which we denote by  
\begin{equation}\label{sheaf rest}
  - \otimes_{\scr{R}} \scr{A}:Mod(\scr{R}) \leftrightarrows Mod(\scr{A}): U.  
\end{equation}
 
\begin{definition}
We say that $\scr{F} \in \scr{R^{+}}$-Mod is \textit{strictly perfect} if it is bounded and a direct summand of a free sheaf of finite rank in each degree. A dg $\scr{R}$-module $\scr{G}$ is \textit{perfect}
if for every $x \in X$ there is a neighbourhood $U$ such that $\scr{G}_{U}$ is quasi-isomorphic to
a strictly perfect dg sheaf. We say that a perfect sheaf is \textit{globally bounded} if there exist integers $a, b$ and $N$ such that there exists an open cover $\{ U_{i}\}$ such that $\scr{G}_{U_{i}}$ is quasi-isomorphic to a strictly perfect sheaf concentrated in degrees $[a, b]$ and has, at most $N$ generators. We use the notation $D_{perf}(\scr{R})$ and $D^{B}_{perf}(\scr{R})$ to denote the subcategories of $D(\scr{R})$ consisting of perfect and bounded perfect dg sheaves of $\scr{R}$-modules, respectively. 
\end{definition}

\begin{definition}\label{clc def}
    Let $\underline{k}^{+}$ denote the sheafification of the constant presheaf $\underline{k}$ on $X$. We say that a sheaf of dg $\underline{k}^{+}$-vector spaces is \textit{cohomologically locally constant (clc)} if it is locally quasi-isomorphic to a free sheaf of dg $\underline{k}^{+}$-vector spaces. 
\end{definition}

\begin{remark}
Note that the cohomology of clc sheaf of dg $\underline{k}^{+}$-vector spaces is locally constant.  Moreover,  under the assumption that $X$ is locally contractible, the converse also holds true. In particular, for a locally contractible space the derived category of locally free dg $\underline{k}^{+}$-vector spaces is equivalent to the derived category of clc sheaves, see \cite[Lemma 7.12]{chuang2021maurer}. 
\end{remark}

\subsection{The de Rham algebra} \label{de Rham}

Let $X$ be a connected smooth manifold. We take the field $k$ to be $\mathbb{R}$, and denote by $\scr{R} = \ubb $ will denote the constant presheaf on $X$. Let $(\scr{A}, d)=(\Omega, d)$ denote sheaf of de Rham algebras on $X$. By the Poincar\'e Lemma, $\underline{\mathbb{R}}^{+} \lra \Omega $ is a quasi-isomorphism, in particular the assumptions of Theorem \ref{manin theorem final} are satisfied.  We endow the category of dg $\Omega(X)$-modules over the de Rham algebra $\Omega(X)$ with the model structure of second kind. The category of presheaves of dg $\Omega$-modules is endowed with the local model structure, where the sheaf de Rham algebras is treated as a presheaf. Let $K$ denote the set 

\begin{equation*}
    \{ p^{*}(T)| T \in TW_{fg}(\Omega(X)) \}
\end{equation*}
and let $L$ denote the set 

\begin{equation*}
    \{J ((\scr{K})^{fib})^{cof}| \scr{K} \in K\}.
\end{equation*}

The right Bousfield localisation of  $PMod(\Omega)$  at $K$ is denoted by $R_{K}PMod(\Omega)$. Denote the right Bousfield localisation of  $ PMod(\underline{\mathbb{R}})$  at $L$ by $R_{L}PMod(\underline{{\mathbb{R}}})$.  Then by Theorem \ref{manin theorem final} we have the following theorem.

\begin{theorem}\label{zig-zag de Rham}
    Let $X$ be a smooth connected manifold. Then the adjunctions

   \[
\begin{tikzcd}[column sep=large]
\Omega(X)-Mod \arrow[r, shift left, "p^{*}"]  &  R_{K} PMod(\Omega) \arrow[l, shift left, "p_{*}" ]  \arrow[r, shift right, swap, "J" ] & R_{L}PMod(\ubb) \arrow[l, shift right,swap, "-\ot_{\ubb} \Omega" ]
\end{tikzcd}
\]   
form a zig-zag of dg Quillen equivalences.

\end{theorem}

Using the same notation as in equations (\ref{sheaf inverse})  and (\ref{sheaf rest}), define
\begin{equation*}
    F=  U\circ q^{*}: TW_{perf}(\Omega(X)) \lra Mod(\ubb^{+}). 
\end{equation*}
According to \cite[Tehorem 8.1]{chuang2021maurer},  the functor $F$ induces an equivalence
\begin{equation*}
   H^{0}(TW_{perf}(\Omega(X))) \lra D_{perf}(\underline{\mathbb{R}}^{+}). 
\end{equation*}
As highlighted in  \cite[Remark 8.1]{chuang2021maurer}, perfect twisted modules over the de Rham algebra correspond to dg vector bundles equipped with flat  $\mathbb{Z}$-graded connections. Furthermore, perfect dg $\ubb^{+}$-modules are clc sheaves whose cohomology sheaves are locally constant and of finite rank. Consequently, this equivalence can be seen as a natural derived analogue of the classical Riemann-Hilbert correspondence, which relates flat vector bundles to locally constant sheaves. 
 Furthermore, the the functor  $H^{0}(TW_{perf}(X)) \lra D_{perf}(\underline{\mathbb{R}})$ can be represented as the following composition:

 \[
\begin{tikzcd}
H^{0}(TW_{perf}(X)) \arrow[r, "q^{*}"] &  H^{0}(\Omega-Mod) \arrow[r, "Q"] & D(\Omega) \arrow[r, "\mathbb{R}U"] &  D_{perf}(\ubb)
\end{tikzcd}
\]

where $Q$ is the quotient by quasi-isomorphisms. In particular, up to local weak equivalence, we can identify $L$ with the image of $TW_{fg}(\Omega(X))$  under $F$.  Let $\mathfrak{L} $ denote the triangulated subcategory of $D(\underline{\mathbb{R}})$ given by $ ^{\perp}(L^{\perp})$.  Note that this is a thick localising subcategory which contains $D_{perf}(\underline{\mathbb{R}})$.

\begin{corollary}
\label{quotient cor}
The following triangulated categories are equivalent: 
    \begin{enumerate}[label=(\roman*)]
        \item  $D^{II}(\Omega(X))$,
        \item $Ho(R_{L}Mod(\ubb))$,
        \item $\mathfrak{L}$.
    \end{enumerate}  
    Furthermore, the category $\mathfrak{L}$ is the minimal triangulated subcategory containing $L$.
\end{corollary}

\begin{proof}
The first two are equivalent by Theorem \ref{zig-zag de Rham}. By Remark \ref{triangle stuff} , 

\begin{align*}
    Ho(R_{L}Mod(\ubb))  \cong  D(\ubb)/L^{\perp}
\end{align*}
and the quotient functor $D(\ubb) \lra D(\ubb)/L^{\perp}$ is given by the right derived functor of the identity $\mathbb{R}I$. Thus by \cite[Proposition 4.9.1.]{krause2008localization} the composition with the inclusion 

\[
\begin{tikzcd}
^{\perp}(L^{\perp}) \arrow[r, "i"]   
& D(\ubb) \arrow[r]
&D(\ubb)/L^{\perp} .
\end{tikzcd}
\]
is an equivalence.
By Remark \ref{triangle stuff}, we can conclude that  $\mathfrak{L} $ is equivalent to the kernel of an exact localisation functor. By \cite[Proposition 4.10.1.  ]{krause2008localization} $^{\perp}(\mathfrak{L}^{\perp}) \cong \mathfrak{L}$, and it follows that  $\mathfrak{L}$ is the minimal triangulated subcategory containing $L$.

\end{proof}

\subsection{The Dolbeault algebra}\label{dol}

In this subsection, we set $k =\mathbb{C}$ and let $X$ be a holomorphic manifold equipped with the sheaf of holomorphic functions $\scr{O}$. Let $A^{0 *}(X)$ denote the Dolbeault algebra over $X$ and, $\scr{A}^{0*}$ denote the sheaf of holomorphic forms. Analogous to the previous subsection, we equip $A^{0 *}(X)-Mod$ with the model structure of second kind, while  $PMod(\scr{A}^{0*})$ and $PMod(\scr{O})$  are endowed the local model structure. Recall, that $\scr{A}^{0*}$ provides a fine resolution of the sheaf of holomorphic functions $\scr{O}$, as such $\scr{O} \lra \scr{A}^{0*}$ is a local weak equivalence. As in the previous subsection, let $K$ denote the set

\begin{equation*}
    \{ p^{*}(T)| T \in TW_{fg}(A^{0*}(X)) \},
\end{equation*}
 and $L$ denote the set

\begin{equation*}
    \{J ((\scr{K})^{fib})^{cof}| \scr{K} \in K\}.
\end{equation*}

\begin{theorem}\label{zig-zag dol}
   Let $X$ be a holomorphic manifold. Then the adjunctions

\[
\begin{tikzcd}[column sep=large]
A^{0*}(X)-Mod \arrow[r, shift left, "p^{*}"]  &  R_{K} PMod(\scr{A}^{*0}) \arrow[l, shift left, "p_{*}" ]  \arrow[r, shift right, swap, "J" ] & R_{L}PMod(\scr{O}) \arrow[l, shift right,swap, "-\ot_{\scr{O}} \scr{A}^{*0}" ]
\end{tikzcd}
\]

form a zig-zag of dg Quillen equivalences.

\end{theorem}
\begin{proof}
    This follows from Theorem \ref{manin theorem final}.
\end{proof}

In \cite[Theorem 8.3]{chuang2021maurer} it is shown that the composition $U \circ q^{*}: A^{0 *}(X)-Mod \lra  Mod(\scr{O})$ induces an equivalence of  categories $TW_{perf}(A^{0 *}(X)-Mod) \cong D^{b}_{perf}(\scr{O}).$ Let $ \mathfrak{Q} $ denote localising triangulated subcategory of $ D(\scr{O})$ defined by $^{\perp}(L^{\perp})$. By the same reasoning as Corollary \ref{quotient cor} we obtain the following result. 

\begin{corollary}\label{main 2}
    Let $\mathfrak{Q}$ denote the closure of $D^{b}_{perf}(\scr{O})$ with respect to all direct sums. The following triangulated categories are equivalent: 
    \begin{enumerate}[label=(\roman*)]
        \item  $D^{II}(\scr{O})$,
        \item $Ho(R_{L}PMod(\scr{O}))$,
        \item $\mathfrak{Q}$.
    \end{enumerate}
    Furthermore, the category $\mathfrak{Q}$ is the minimal triangulated subcategory containing $L$.
\end{corollary}

 \section{The Singular Cochain Algebra}

In this section, we develop singular analogues of the theorems presented in subsections  \ref{de Rham} and \ref{dol}. We assume that $X$ is a connected, locally contactable topological space and that $k$ is a field of characteristic zero.  The normalized singular cochains on $X$  with coefficients in $k$ will be denoted by $C^{*}(X)$. 
Since we treat $C^{*}(X)$ as a pseudo-compact dg algebra, we begin by recalling the relevant information about pseudo-compact algebras and their duality with coalgebras.

\subsection{Contramodules}

A pseudo-compact (dg) $k$-vector space is a cofiltered limit of finite-dimensional (dg) $k$-vector spaces, equipped with the inverse limit topology.  Morphism between pseudo-compact vector spaces are assumed to be continuous with respect to the inverse limit topology. Moreover, the category of pseudo-compact vector spaces possesses the structure of a monoidal category with respect to the completed tensor product. Specifically, given $ V = \varprojlim V_{\alpha}$ and $U= \varprojlim U_{\beta}$ where $V_{\alpha}$ and $U_{\beta}$ are finite-dimensional (dg) vector spaces, the completed tensor product is defined as: 

\begin{equation*}
    V \wot U : = \varprojlim V_{\alpha} \ot U_{\beta}
\end{equation*}
A pseudo-compact (dg) algebra is defined as an associative monoid in the monoidal category of pseudo-compact (dg) vector spaces. Equivalently, a pseudo-compact (dg) algebra can be understood as a pro-object in the category of finite dimensional (dg) algebras. The fundamental theorem of coalgebras allows us to identify the category of (dg) coalgebras with the category of ind-objects in the category of finite dimensional (dg) coalgebras. The anti-equivalence between the category of finite dimensional (dg) coalebras and finite dimensional (dg) algebras given by taking the linear continuous dual induces an equivalence between (dg) coalgebras and the opposite category of pseudo-compact (dg) algebras.

Additionally, we need to define the tensor product of a pseudo-compact vector space $ V = \varprojlim V_{\alpha}$ and a discrete $U$ vector space. This tensor product is given by
\begin{equation*}
    V \wot U : = \varprojlim V_{\alpha} \ot U.
\end{equation*}
A more detailed treatment of pseudo-compact dg algebras can be found in \cite{bergh2010calabi}.

\begin{definition}
    Let $A$ be a pseudo-compact algebra. A \textit{right $A$-contramodule} is a discrete $k$-vector space equipped, $Q$ with a \textit{contraaction map} $\pi_{Q}:Q \widehat{\ot} A \lra Q$ satisfying the usual associativity and unitality conditions.

\end{definition}
Equivalently,  a contramodule over a $k$-coalgebra $C$ can be defined as a vector space with a contraaction map $\pi_{Q}:Hom_{k}(C, Q) \lra Q$ satisfying associativity and unitalitly conditions. This is equivalent to the definition above after dualising $A: = C^{*}$ and noting that $Hom_{k}(C, Q) \cong Q \wot C^{*} $.

Since we do not require the contraaction map $\pi_{Q}:Q \widehat{\ot} A \lra Q$ to be continuous, a contramodule cannot be viewed as a module over a monoid in the symmetric monoidal category of pseudo-compact vector spaces. Since we take the completed tensor product $Q \widehat{\ot} A$ a contramodule cannot be interpreted as a module over a monoid in the symmetric monoidal category of discrete vector spaces. 

Naturally, this definition extends to the dg setting. A dg contramodule, Q over a pseudo-compact dg algebra A is defined as a graded contramodule with a differential $d:Q \lra Q $ of degree $1$, where $d^{2}= 0$ and the contraaction map is a morphism of complexes.  Unless otherwise specified we will always work with right (dg) contramodules.

\begin{definition}
Let $T$ be an $A$-contramodule. We say that $T$ is \textit{free} if $T$ is isomorphic to a contramodule of the form $V \wot A$, where $V$ is a $k$-vector space and the contraction on $V \wot A$ given by the action of $A$ on its self. 
\end{definition}

Free $A$-contramodules are those that arise from the image of the free functor which takes a vector space $V$ to $V \wot A $. For any $A$-contramodule, $Q$ there exists natural isomorphism of $k$-vector spaces:

\begin{equation*}
    \uhom_{A}(V \wot A, Q) \cong \uhom_{k}(V, Q).
\end{equation*}
Additionally, an $A$-contramodule is projective as an object in the category of $A$-contramodules if and only if it is the direct summand of a free contramodule. The category of $A$-contramodules admits both products and direct sums, where products are exact. However, direct sums and cofiltered limits are not exact. Consequently, the category of $A$-contramodules is not Grothendieck but it is a locally presentable abelian category with enough projective objects (see, \cite[Section 1.2]{positselski2021contramodules}).

\begin{definition}
Let $A$ be a pseudo-compact dg algebra. We define a right dg $A$-contramodule, $Q$ to be a $\mathbb{Z}$-graded $A$-contramodule equipped with a degree $1$ derivation, $d: Q \lra Q$ with $d^{2}= 0$ such that the contraaction map $\pi_{Q}: Q \widehat{\ot} A \lra Q$ is a morphism of complexes. 
\end{definition}

From this point onwards we will always work with dg contramodules over a dg pseudo-compact algebra $A$. We denote the category of right dg $A$-modules by $A-Ctmod$, and we denote the full subcategory of $A-Ctmod$ consisting of dg $A$-contramodules that are projective as graded contramodules after forgetting the differential by $A-Ctmod_{proj}$. We will call contramodules of this form \textit{graded projective} contramodules.

\begin{remark}

let $f : A \lra B $ be a morphism of pseudo-compact dg $k$-algebras. Define the \textit{restriction of scalars} functor $R_{f}: B-Ctmod \lra A-Ctmod $ to be the functor which takes a $B$-contramodule $Q$ to the same underlying dg vector space as $Q$ with $A$-contraaction  induced by $f$. The restriction of functor is right adjoint to \textit{extension of scalars} functor which takes an $A$-contramodule $R$ to the coequaliser of the pair of morphisms 

    \begin{equation*}
        R \widehat{\ot} A \widehat{\ot} B \rightrightarrows R \widehat{\ot} B.
    \end{equation*}
One morphism comes from the contraaction of $A$ on $R$ and the other comes from the regarding $B$ as a left $A$-contramodule by restricting along $f$. The $B$-contraaction on $ R\ot_{A} B$ is inherited from regarding $B$ as a contramodule over its self. 

\end{remark}

\begin{definition}

A dg $A$-contramodule is called \textit{contraacyclic}  if it belongs to the minimal
triangulated subcategory of the homotopy category of dg $A$-contramodules containing the total dg $A$-contramodules of exact triples of dg $A$-contamodules and closed under infinite products. 
\end{definition}

\begin{remark}\label{contra remark}
 The contraderived category of dg $A$-contramodules which we will denote by $D^{ct}(A)$ is defined to be the Verdier quotient of the cochain homotopy category by the thick subcategory of contraacyclic dg $A$-contramodules. Let $K(A-Ctmod)$  denote the cochain homotopy category and let $Ac^{ct}(A-Ctmod) \subset K(A-Ctmod) $ denote the thick subcategory of contraacyclic  contramodules. Then the Verdier quotient 

\begin{equation*}
    D^{ct}(A) : =  K(A-Ctmod)/Ac^{ct}(A-Ctmod)
\end{equation*}
is called the contraderived category of dg $A$-contramodules.  Furthermore, there is a triangulated equivalence between the cochain homotopy category of graded projective contramodules and the contraderived category, that is, 

\begin{equation*}
    K( A-Ctmod_{proj})\cong D^{ct}(A). 
\end{equation*}

In addition to above description the contraderived category can be realised the homotopy category of the following model category structure.

We also note that the contraderived category is compactly generated. To see this we recall that the coderived category of dg comodules over a coalgebra is compactly generated by the set of finite dimensional comodules. Then the derived co-contra correspondence implies that the contraderived category of contramodules over a coalgebra is also compactly generated  (See \cite[Proof of section 2.4]{positselski2011two}).

\begin{theorem}\cite[Theorem 8.2 (b)]{positselski2011two}
 There exists a cofibrantly generated model category structure on the category $A-Ctmod$, where a morphism $f: Q \lra R$ is defined to be:

\begin{enumerate}[label=(\roman*)]
\item a weak equivalence if its cone is a contraacyclic $A$-contramodule,
\item a fibration if it is surjective, 
\item a cofibration if it is injective and its cokernel belongs to $A-Ctmod_{proj}$.
 \end{enumerate}
 Moreover, all object in this model structure are fibrant and the cofibrant objects are precisely the objects which belong to $A-Ctmod_{proj}$. 
\end{theorem}

\end{remark}
\begin{lemma}\label{monadic lemma 1}
Let $( A-Ctmod)_{fr}$ denote the full subcategory of  $A-Ctmod$ that are free as dg contramodules.   Let $\mcl{C}$ and $\mcl{D}$ be categories  Let $i:( A-Ctmod)_{fr} \lra A-Ctmod$ denote the inclusion of free contramodules into the category of contramodules then any diagram of the form

\[
\begin{tikzcd}
A-Ctmod_{fr} \arrow{r}{F} \arrow{d}{i} & \arrow{d}{H} \mcl{C} \\
A-Ctmod \arrow{r}{G}& \mcl{D}
\end{tikzcd}
\]
admits a lift $A- Ctmod \lra C$.
\end{lemma}

\begin{proof}
Firstly, note that the free-forgetful adjunction $-\widehat{\ot} A: k-Mod \leftrightarrows A-Ctmod: U $ is monadic and let $T= 
  - \widehat{\ot} A \circ U$.  In particular, Beck's monadicity theorem implies any contramodule, $Q$, can obtained as a split coequaliser 
\begin{equation*}
     Q \cong coeq[T^{2} Q \rightrightarrows T Q  ]. 
 \end{equation*}
Let $Q^{*}$ denote this coequaliser diagram, which can be regarded as a diagram in the category of free contamodules, hence $Q =  \colim  i(Q^{*})$.  Since split coequalisers are absolute colimits any contramodule can be obtained as an absolute colimit of free contramodules.  We now show that $\bar{G}(Q) : = \colim  FQ^{*}$ is the desired lift. Since the colimit is absolute $H(\bar{G}(Q)) \cong  \colim  H (F (Q^{*})) \cong \colim  G(i (Q^{*}) ) \cong G(Q)$ and $F(Q) = \colim  F( i(Q^{*}))  \cong \bar{G}(Q)$.
\end{proof}

\begin{lemma} \label{monadic lemma 2}
Let $\mcl{C}$ be a category and $F, G: A-Ctmod \lra \mcl{C}$ be two functors. If there is a natural transformation $\eta: F \Rightarrow G$ which is an isomorphism when restricted to free contramodules then $\eta$ is an isomorphism for all contramodules. 
\end{lemma}

\begin{proof}

  With the notation of the previous proof, we have a natural transformation  $\eta i :  F i\Rightarrow G i $ which gives a natural transformation of diagrams $ F(i(Q^{*}) \lra G(i(Q^{*}))$, the colimit of which is $\eta_{Q}$. Since $\eta i$ is an isomorphism $\eta_{Q}$ must also be an isomorphism. 

\end{proof}

\begin{proposition}
    The category $A-Ctmod$ is tensored and cotensored over the category of dg $k$-vector spaces. 
\end{proposition}

\begin{proof}

Define the tensor $V \ot^{ct} Q $ to be the tensor product of the underlying dg vector spaces with contraaction induced by the right contraaction of $A$ on $Q$. Define the cotensor $[V, Q]_{ct}$ to be the dg vector space $\uhom_{k}(V, Q)$ with contraaction defined as follows. Let $f \in \uhom_{k}(V, Q), \  v \in V$ and $a \in A$ then the contraaction  $\uhom_{k}(V, Q) \widehat{\ot} A \lra \ \uhom_{k}(V, Q)$ is defined by $f(v) \ot a = f(v)a $.  The tensor and cotensor should satisfy the following adjunctions:

\begin{align*}
    \uhom_{A}(V \ot^{ct} Q, R) &\cong \uhom_{k}(V, \uhom_{A}(Q, R)) \\
                            &\cong \uhom_{A}(Q, [V, R]_{ct}). 
\end{align*}
By Lemma \ref{monadic lemma 1} and Lemma \ref{monadic lemma 2} it is enough to show that the adjunctions hold for free contramodules, let assume that $Q = W \wot A $ and $R = Z \wot A$ are free contramodules. Then 
  
\begin{equation*}
    \uhom_{A}(V \ot^{ct} Q, R ) \cong eq[V \ot_{k} W \widehat{\ot } A \widehat{\ot} A \rightrightarrows Z \wot A]
\end{equation*}
computed in dg vector spaces. Then the usual tensor hom adjunction for dg vector spaces implies that 
\begin{align*}
     eq[V \ot_{k} W \wot A \widehat{\ot} A \rightrightarrows Z \wot A ] &\cong eq[ W \wot A\widehat{\ot} A \rightrightarrows \uhom_{k}(V , Z \wot A)]\\
                                                       &\cong \uhom_{A}(Q, [V, R]_{ct}).   
\end{align*}
On the other hand 
\begin{align*}
    \uhom_{k}(V, \uhom_{A}(Q, R)) &\cong eq[\uhom_{k}(V, W \wot A \wot  A) \rightrightarrows \uhom_{k}(V, Z \wot A)]\\
                                   &\cong eq[\uhom_{k}(V, W \wot A \wot  A) \rightrightarrows \uhom_{k}(V, Z \wot A)]\\
                                   &\cong \uhom_{k}(V, \uhom_{A}(W \wot A, R))\\
                                   &\cong \uhom_{k}(V, \uhom_{k}(W, U R))\\
                                   & \cong \uhom_{k}(V \ot W, UR))\\
                                   &\cong \uhom_{A}(V \ot W \wot A, R)\\
                                   & \cong \uhom_{A}(V \ot^{ct} Q, R).
\end{align*}
\end{proof}

\begin{lemma}\label{basechangecontra}
    The restriction and extension of scalars are adjoint and the adjunction is a dg enriched adjunction. 
\end{lemma}

\begin{proof}
By proposition \ref{enriching prop},  it suffices to verify that the functors give rise to an adjunction between the underlying categories and the left adjoint which preserves tensors. Let $-\ot^{ct}_{A}-$ be the  on $A-Ctmod$  and $-\ot^{ct}_{B}-$ be the tensor on $B-Ctmod$. Let $V$ be a dg vector space and $Q$ an $A$-contramodule. Then $(V \ot^{ct}_{A} Q) \ot_{A} B $ is precisely the dg vector space $ V \ot_{k} Q \ot_{A} B$ with contraaction given by the contraaction of $B$ on its self. Similarly, $V \ot^{ct}_{B} (Q \ot_{A} B) $ is precisely the dg vector space $ V \ot_{k} Q \ot_{A} B$ with contraaction given by the contraaction of $B$ hence the extension of scalars functor preserve tensors. 
\end{proof}

\begin{proposition}
    With the model structure above $A-Ctmod$ is a stable dg model category. 
    
\end{proposition}

\begin{proof}
  An analogous argument to Remark \ref{stability} shows that the suspension object associated to a cofibrant object will be the usual shift to the right by one degree, and similarly, the loop object associated to a fibrant object will be the usual shift to the left by one degree.
  
To see that $A-Ctmod$ is a dg model category, we need to verify the push out product axiom. Let $I'$ and $J'$ be as in Proposition \ref{comod is dg}. Consider a morphism of contramodules $f: Q \lra R$. We will check the following:

\begin{enumerate}[label=(\roman*)]
    \item if $i \in I'$ and $f$ is a cofibration then $f \square i$ is a cofibration.
          \item if $i \in I'$ and $f$ is an acyclic cofibration then $f \square i$ is an acyclic cofibration.
                \item if $j \in J' $ and $f$ is a cofibration then $f \square i $ is an acyclic cofibration.
\end{enumerate}

For (iii) we have $j \square f : D^{n} \ot^{ct} Q \lra D^{n} \ot^{ct} R $, which is injective and has cokernel $D^{n} \ot^{ct} coker(f)$, which is cofibrant since the cokernel of $f$ is cofibrant. To see that $j \square f$ is a weak equivalence, we will show that $Cone(j \square f)$ is contractible and therefore contraacyclic. Firstly, note that for any contramodule, $P$ the complex $D^{n} \ot^{ct} P $ is given by mapping cone of $Id_{P}$ with a degree $n$ shift. Since the mapping cone of $Id_{P}$ is contractible, so is $D^{n} \ot^{ct} P $. Since $Cone(j \square f) = D^{n} \ot^{ct} Cone(f)$ the result follows.

For (i) let $\alpha$ denote the pushout of the map $S^{n-1} \ot^{ct} f$ and let $\beta$ denote the pushout of the map $i \ot^{ct} Q$. It follows that $\alpha$ is a cofibration since pushouts of cofibrations are cofibrations and $S^{n-1} \ot^{ct} f$ a cofibration. We can also conclude that $\beta$ is an injection since the pushouts of injections are injections.  Let $P$ denote the pushout then it follows that the pushout product map $i \square f: P \lra D^{n} \ot^{ct} R$ is an injection.  It remains to check that the cokernel of $i \square f$ is cofibrant, consider the following diagram:

 \[
\begin{tikzcd}
 & 0  \arrow[d] & 0  \arrow[d] & 0 \arrow[d] \\
0 \arrow[r] & S^{n-1} \ot^{ct} R \arrow[r, "\beta"] \arrow[d] &
P \arrow[d] \arrow[r] & coker(\beta) \arrow[d, "\gamma"]  \arrow[r] & 0 \\
0 \arrow[r] & S^{n-1} \ot^{ct} R \arrow[r, "i \ot^{ct} R"] \arrow[d] & D^{n}\ot^{ct} R \arrow[r] \arrow[d] & coker(i \ot^{ct} R) \arrow[r] \arrow[d] & 0\\
0 \arrow[r] & 0 \arrow[r] \arrow[d] & coker(i \square f) \arrow[r] \arrow[d] & coker(\gamma)  \arrow[d] &\\
 & 0 & 0 & 0. 
\end{tikzcd}
\]
Since pushouts preserve cokernels, the cokernel of $\beta$  is isomorphic to $ coker (i \ot^{ct} Q)  \cong S^{n} \ot^{ct} Q$ and $coker(i \ot^{ct} R) \cong S^{n} \ot^{ct} R$ and $\gamma$  is given by $S^{n} \ot^{ct} f$.  By the 3x3 Lemma, the bottom row is exact and $coker(i \square f) \cong coker(\gamma)$, which is cofibrant.

Since the pushout of an acyclic cofibration is an acyclic cofibration and $D^{n} \ot^{ct} f$ is a weak equivalence, (ii) follows from the two out of three axiom.

\end{proof}

\begin{proposition}\label{we contra}
   Let $\mcl{C}$ denote the set of compact generators in the contraderived category $D^{ct}(A)$. The A morphism $f: Q \lra R$ of $A$-contramodules is a weak equivalence if and only if the induced map $\rdh_{A}(C, f)$ is a quasi-isomorphism for all $C \in \mcl{C}$. 
\end{proposition}

\begin{proof}
Let $S$ be a contramodule. It is sufficient to show that a $S$ is contraacyclic if and only if $\rdh_{A}(C, S)$ is a quasi-isomorphic to $0$  for all $C \in \mcl{C}$.  If $S$ is contraacyclic, then the graded abelian groups $\underline{Hom}_{D^{ct}(A)}(C, S)$ are trivial, and the result follows from $H^{*}(\rdh(C, S)) \cong \underline{Hom}_{D^{ct}(A)}(C, S) $. The converse follows from \cite[Lemma 2.2.1]{schwede2003stable}. 

\end{proof}

\subsection{Presheaves of Contramodules}

\begin{definition}
Let $\scr{A}$ be a presheaf of pseudo-compact dg $k$-algebras on $X$. A presheaf of dg $\scr{A}$-contramodules is a presheaf of dg $k$-vector spaces $\scr{F}$ such that for any open set $U \subset X$,  $\scr{F}(U)$ is a $\scr{A}(U)$-contramodule and the restriction maps are compatible with the contramodule structure.  We denote the category of presheaves of dg $\scr{A}$-contramodules by $PCtmod(\scr{A})$. 
\end{definition}

\begin{remark}\label{tensor and cotensor for ctpsh}
It is clear that the category  $PCtmod(\scr{A})$ is a dg category. Furthermore,  $PCtmod(\scr{A})$ is tensored and cotensored over dg $k$-vector spaces. Let 

\begin{align*}
    - \ot^{p} - : k-Mod \times PCtmod(\scr{A}) \lra PCtmod(\scr{A})                   
\end{align*}

denote the tensor of $PCtmod(\scr{A})$ over $k-Mod$. The tensor is defined to be the functor which sends $(V, \scr{F}) $  in  $k-Mod \times PCtmod(\scr{A})$ to the presheaf defined by 

\begin{equation*}
    U \longmapsto V \ot_{ct} \scr{F}(U),
\end{equation*}
for each open $U \subset X$,  were $ V \ot_{ct} \scr{F}(U)$   denotes the tensor on contramodules.  Let

\begin{align*}
    [-, - ]_{p}: k-Mod^{op} \times PCtmod(\scr{A}) \lra PCtmod(\scr{A})                   
\end{align*}

denote the cotensor.  For an open set $U  \subset X$, the cotensor is given by 
\begin{equation*}
    U \longmapsto \uhom_{\underline{k}}(\underline{V}_{U}, \scr{F}_{U}). 
\end{equation*}
For $W \subset U $ the contraaction map,

\begin{equation*}
    \uhom_{\underline{k}}(\underline{V}_{U}(W), \scr{F}_{U}(V)) \wot \scr{A}(W) \lra \uhom_{\underline{k}}(\underline{V}_{U}(W), \scr{F}_{U}(V)),
\end{equation*}
is defined on elements by $f(v) \wot a = f(v)a$.

\end{remark}

\begin{lemma}\label{base change for pct}
   Let $\scr{A}$ and $\scr{B}$ be presheaves of pseudo-compact dg $k$-algebras on $X$ and $p: \scr{A} \lra \scr{B}$ a morphism of presheaves of pseudo-compact dg $k$-algebras. Then $p$ induces a dg adjunction $
- \ot_{\scr{A}}\scr{B}:PCtmod(\scr{A}) \rightleftarrows  PCtmod(\scr{B}): R_{p} $.

\end{lemma}

\begin{proof}
The adjunction between the underlying categories follows from the restriction and extension adjunction for contramodules. It is immediate from the definitions that the left adjoint preserves tensors. Therefore, the adjunction is a dg adjunction by Proposition \ref{enriching prop} . 
     
\end{proof}

Let $p: X \lra Y$ be a continuous map of topological spaces. Then as usual we define the pushforward of a presheaf, $\scr{F}$  on $X$ valued in any category,  to be the  presheaf  which takes an open set $V \subset Y $ to $\scr{F}(p^{-1}(V))$ and we denote the pushforward of $\scr{F} $ along $p$ by $p_{*}(\scr{F})$. Let $\scr{A}_{X}$ a presheaf of pseudo-compact dg $k$-algebras on $X$, then the pushforward functor defines a functor $p_{*}:PCtmod(\scr{A}) \lra PCtmod(p_{*}\scr{A})$. 

Let $\scr{B}_{Y}$ be a presheaf of pseudo-compact dg $k$-algebras on $Y$. Define the pullback of a presheaf $\scr{G}$ along $p$ to be the presheaf determined by the rule 
\begin{equation*}
    p^{\#} \scr{G}(U) := \colim_{V} \scr{G}(V)
\end{equation*}
where the colimit runs over all open $V \subset Y$ containing $p(U)$. The pullback determines a functor $p^{\#}: \scr{B}_{Y}-PCtmod \lra p^{\#}\scr{B}_{Y}-PCtmod$. 
\begin{definition}\label{p map def}
    Let $p : X \lra Y $ be continuous map between topological spaces and $\scr{A}_{X}$ be a presheaf of pseudo-compact dg $k$-algebras on $X$ and $\scr{B}_{Y}$ be a presheaf of pseudo-compact dg $k$-algebras on $Y$.  We say that $\widehat{p}: \scr{B}_{Y}  \lra \scr{A}_{X}$ is a $p$\textit{-map} if there exists a collection of maps $\widehat{p}_{V}: \scr{B}_{Y}(V) \lra   \scr{A}_{X} (p^{-1}(V)) $  indexed by the open sets $V \subset Y $ that  are compatible with the restriction maps.

\end{definition}
With notation as above we define the \textit{direct image functor }  $p_{*}: \scr{A}_{X}-PCtmod \lra \scr{B}_{Y}-PCtmod$ to be the functor which takes $\scr{F}$ to the presheaf  $p_{*}(\scr{F})$ with contramodule structure given by the restriction via $\widehat{p}: \scr{B}_{Y}  \lra \scr{A}_{X}$.  We define the \textit{inverse image functor} $p^{*}: PCtmod(\scr{B}_{Y}) \lra PCtmod(\scr{A}_{X})$ to be the functor which which sends a presheaf $ \scr{G}$ 

\begin{equation*}
   \scr{G} \mapsto p^{\#}\scr{G} \ot_{p_{\#}\scr{B}_{Y}} \scr{A}_{X}
\end{equation*}

\begin{proposition}\label{direct inverse image}
    Let $p: X \lra Y $  and  $\widehat{p}: \scr{B}_{Y}  \lra \scr{A}_{X}$ be as in Definition \ref{p map def} then  the functor $p^{*}$ is left adjoint to the functor $p_{*}$  and the adjunction is a dg adjunction. 
\end{proposition}
\begin{proof}
Let $\scr{G}$ be a presheaf of  $\scr{B}_{Y}$-contramodules and $\scr{F}$ a presheaf of $\scr{A}_{X}$-contramodules. Then $p^{\#}\scr{G}$ is a presheaf of $p^{\#}\scr{B}_{Y}$-contramodules.  Standard arguments regarding inverse direct image functors (see, e.g. \cite[Section 6.9]{bosch2013algebraic}) give a natural isomorphism  
    \begin{equation*}
    Hom_{p^{\#}\scr{B}_{Y}} (p^{\#}\scr{G}, \scr{F}) \cong Hom_{\scr{B}_{Y}} (\scr{G}, p_{*}\scr{F}). 
\end{equation*}

Then by  Lemma \ref{base change for pct}  the map the morphism  $\widehat{p}$ induces a natural isomorphism 

\begin{equation*}
    Hom_{p^{\#}\scr{B}_{Y}} (p^{\#}\scr{G}, \scr{F}) \cong Hom_{\scr{A}_{X}} (p^{\#}\scr{G} \ot_{p^{\#}\scr{B}} \scr{A}_{X},\scr{F}). 
\end{equation*}

From the definitions, it is clear that $p^{*}$ preserves tensors, and we can conclude that the adjunction is a dg adjunction. 

\end{proof}

\begin{remark}\label{compact remark}
Recall that the category of presheaves of dg vector spaces $PMod(\underline{k})$ is locally finitely presentable, with a set of compact generators given by the presheaves $\underline{k}_{U}$ for $U \subset X$ open. Let $S$ denote the set compact generators. The set of all compact objects is given by the closure of this set with respect to all finite colimits; denote this set by $\overline{S}$.

\end{remark}

 Let  $\scr{A}$ be a presheaf of pseudo-compact dg $k$-algebras and assume there exists a stalk-wise quasi-isomorphism $\underline{k} \lra \scr{A}$. By  Lemma \ref{base change for pct} we have an adjunction $-\ot_{\underline{k}} \scr{A} : PMod(\underline{k}) \rightleftarrows PCtmod(\scr{A}): J$.

\begin{lemma} \label{compact objects}
    The functor $-\ot_{\underline{k}} \scr{A} : PMod(\underline{k}) \lra PCtmod(\scr{A})$ preserves compact objects.  
\end{lemma}
    \begin{proof}
Let $\scr{S}$  be a compact object in   $   PMod(\underline{k}) $. By Remark \ref{compact remark}  $\scr{S} $ is of the form $ \colim_{I} \underline{k}_{U}^{i}$ for some finite indexing category $I$.  Since  $- \ot_{\underline{k}} \scr{A}$ preserves colimits and $\underline{k}_{U} \ot_{\underline{k}} \scr{A} \cong \scr{A}_{U}$  is compact we can conclude that $\scr{P} \ot_{\underline{k}} \scr{A}$  is isomorphic to a finite colimit of compact objects. 
    \end{proof}

\begin{proposition}
\label{pshctmodmodel}
    There exists a right proper, combinatorial, stable, dg model category structure on the category of presheaves of dg $\scr{A}$-contramodules. Let $f: \scr{F} \lra \scr{G}$ be a morphism in $PCtmod(\scr{A})$. Then 
    
    \begin{enumerate}[label=(\roman*)]
    \item  $f$ is a weak equivalence if and only if $f(U): \scr{F}(U) \lra \scr{G}(U)$ is a quasi-isomorphism for all open $U \subset X $,
           \item $f$ a fibration if and only if $f(U): \scr{F}(U) \lra \scr{G}(U)$ is surjective for all open $U \subset X $.         
\end{enumerate}
We will refer this model structure as the global projective model structure. 
\end{proposition}

\begin{proof}
To prove existence we appeal to Theorem \ref{right transfer}. By Lemma \ref{compact objects}, the functor $-\ot_{\underline{k}} \scr{A} $ preserves compact objects and, $PCtmod(\scr{A})$ has functorial path objects for fibrant objects since all objects are fibrant.  Let $(\scr{F}, d)$ be a presheaf of contramodules and let $\scr{I}$ be the complex with $ \scr{A}  \oplus \scr{A}  $ in degree zero and $\scr{A}$  in degree one with differential represented by the matrix 

\begin{equation*}
\begin{bmatrix}
- Id & Id \\ 
\end{bmatrix}.
\end{equation*}

We claim that $ \scr{F} \ot \scr{I}$ defines a functorial path object for $\scr{F}$. Indeed, consider the factorisation 

\[
\begin{tikzcd}
 \scr{F} \arrow[r, "e"] &  \scr{F} \ot \scr{I} \arrow[r,"p"]  & \scr{F} \oplus \scr{F}.
\end{tikzcd}
\]
where 
\begin{equation*}
e=\begin{bmatrix}
Id \\
0 \\
Id 
\end{bmatrix},
\end{equation*}
and
\begin{equation*}
p=\begin{bmatrix}
Id & 0 & 0\\
0 & 0 & 0  \\
0 & 0 & Id 
\end{bmatrix}.
\end{equation*}
To conclude we show that $e$ is a weak equivalence. Let $ r: \scr{F} \ot \scr{I} \lra \scr{F} $ be the map defined by the matrix 
\begin{equation*}
r=\begin{bmatrix}
Id & 0 & 0\\ 
\end{bmatrix}.
\end{equation*}

Then $e$ and $r$ exhibit $\scr{F}$ as a retract of $\scr{F} \ot \scr{I}$, implying that $\scr{F} \ot \scr{I} \cong \scr{F} \oplus \coker (e)$. The pushout of $e$ and the zero map is exactly the mapping cone of the identity, which is acyclic, thus $e$ is the a quasi-isomorphism. By Lemma \ref{Tranferlem}, we can conclude that $PCtmod(\scr{A})$ is a dg model category.

\end{proof}

\begin{proposition}
\label{local on pshct}
    There exists a right proper, combinatorial, stable, dg model category structure on the category of presheaves of dg $\scr{A}$-contramodules.  Let $f: \scr{F} \lra \scr{G}$ be a morphism in $PCtmod(\scr{A})$. Then 
    
    \begin{enumerate}[label=(\roman*)]
    \item  $f$ is a weak equivalence if and only if $f(U): \scr{F}(U) \lra \scr{G}(U)$ is a quasi-isomorphism for all open $U \subset X $,
           \item $f$ a fibration if and only if $f(U): \scr{F}(U) \lra \scr{G}(U)$ is surjective for all open $U \subset X $ and the kernel of $f$ is a hypersheaf .            
\end{enumerate}
We will refer to this model structure as the local model structure. 
\end{proposition}
The proof of  Proposition \ref{local on pshct} is similar to the proof of Proposition \ref{pshctmodmodel}. We use the notation $PCtmod^{G}(\scr{A})$ for the global projective model to distinguish  this model structure from the local model structure which we will denote by $PCtmod(\scr{A})$.

\subsection{Quillen Equivalences}

Let us now recall our assumptions and introduce some notation. We assume that $X$ is a connected, locally contractible topological space, and that $k$ is a field of characteristic zero. We denote the pseudo-compact dg algebra of normalised singular cochains on $X$ with coefficients in $k$ by $C^{*}(X)$, and we denote the presheaf of normalised singular cochains on $X$ by $\scr{C}^{*}$.  Since $X$ is locally contractible, the map  $\underline{k} \lra \scr{C}^{*}$ is a quasi-isomorphism at the level of stalks. Let $p: X \lra *$ denote the continuous map from $X$ to the one-point space. Then by Proposition \ref{direct inverse image},  we obtain a dg adjunction:

\begin{equation*}
    p^{*}: C^{*}(X)-Ctmod \rightleftarrows PCtmod(\scr{C^{*}}): p_{*}.
\end{equation*}

Moreover, by  Lemma \ref{base change for pct} we have a dg adjunction

\begin{equation*}
    - \ot_{\underline{k}} \scr{C}^{*}: PMod(\underline{k}) \rightleftarrows PCtmod(\scr{C^{*}}): J.
\end{equation*}

\begin{proposition}\label{sing quill}
    The adjunction $p^{*}: C^{*}(X)-Ctmod \rightleftarrows PCtmod^{G}(\scr{C^{*}}): p_{*}$ is Quillen.
\end{proposition}
\begin{proof}
It is straightforward to see that $p_{*}$ preserves fibrations, so we only need to show that $p_{*}$ preserves acyclic fibrations. First, note that the functor $p_{*}$ takes weak equivalences to quasi-isomorphisms, as $p_{*}$ is exact. Let $f: \scr{F} \lra \scr{G}$ be an acyclic fibration, then $p_{*}(f)$ is a surjective quasi-isomorphism.

Let $i: Q \lra R$ be a cofibration. The cokernel of $i$ is degree-wise projective and in particular, $i$ is a degree-wise split injection. To conclude, we note that an analogous argument to the proof of  \cite[Proposition 2.3.9.]{hovey2007model}  shows that $p_{*}(f)$ has the right lifting property with respect to all cofibrations $i: Q \lra R$.

\end{proof}

Let $\mcl{C}$ denote a set of compact generators for $D^{ct}(C^{*}(X))$. Additionally, we assume that the contramodules in $\mcl{C}$ have been cofibrantly replaced where necessary. In particular, a morphism of contramodules $f: Q \lra R $ is a weak equivalence if and if the induced map $\uhom (C, Q) \lra \uhom(C, R)$ is a quasi-isomorphism for all $C \in \mcl{C}$, by Proposition  \ref{we contra}.  By Proposition \ref{sing quill} and Proposition \ref{enriched localisation} we immediately obtain the following results.

\begin{proposition}\label{pctgl}
 Let $K= \{p^{*}(C) : C \in \mcl{C} \}    $. Then the right Bousfield localisation of $PCtmod^{G}(\scr{C^{*}})$ at $K$ exists and  $R_{K}PCtmod(\scr{C^{*}})$ is a right proper, combinatorial, stable, dg model category.  Let $f: \scr{F} \lra \scr{G}$ be a morphism of presheaves. Then,  
 
 \begin{enumerate}[label=(\roman*)]
    \item $f$ is a weak equivalence if and only if the induced map $\uhom(\scr{K}, f)$ is a quasi-isomorphism for all $\scr{K} \in K$,
    \item $f$ is a fibration if and only if $f(U): \scr{F}(U) \lra \scr{G}(U)$ is surjective for each open set $U \subset X$.
\end{enumerate}
    
\end{proposition}

\begin{proposition}\label{pctloc}
 Let $K= \{p^{*}(C) : C \in \mcl{C} \}    $. Then the right Bousfield localisation of $PCtmod(\scr{C^{*}})$ at $K$ exists and  $R_{K}PCtmod(\scr{C^{*}})$ is a right proper, combinatorial, stable, dg model category.  Let $f: \scr{F} \lra \scr{G}$ be a morphism of presheaves. Then,  
 
 \begin{enumerate}[label=(\roman*)]
    \item $f$ is a weak equivalence if and only if the induced map $\rdh(\scr{K}, f)$ is a quasi-isomorphism for all $\scr{K} \in K$,
    \item $f$ is a fibration if and only if $f(U): \scr{F}(U) \lra \scr{G}(U)$ is surjective for each open set $U \subset X$ and kernel of $f$ is a hypersheaf. 
\end{enumerate}

\end{proposition}

\begin{lemma}\label{silly lem 2}
    The model structures in Proposition \ref{pctgl}  and in Proposition \ref{pctloc} are Quillen equivalent and the identity functor $Id: R_{K} PCtmod^{G}(\scr{C}^{*}) \lra  R_{K}PCtmod(\scr{C}^{*}) $ is left Quillen.
\end{lemma}
\begin{proof}
This follows from an analogous argument to the proof of Proposition \ref{silly prop} . 

\end{proof}

\begin{proposition}\label{sing qeq 1}
    The adjunction 
 \begin{equation*}
     p^{*}: C^{*}(X)-Ctmod \rightleftarrows R_{K}PCtmod(\scr{C^{*}}):p_{*}
 \end{equation*}   
is a Quillen equivalence. 
\end{proposition}
\begin{proof}
By Lemma \ref{silly lem 2} it is sufficient to show that

\begin{equation*}
    p^{*}: C^{*}(X)-Ctmod \rightleftarrows R_{K}PCtmod^{G}(\scr{C^{*}}):p_{*} 
\end{equation*}
is a Quillen equivalence. Let $f: \scr{F} \lra \scr{G}$ be a morphism of presheaves.  Then $\uhom(\scr{K}, f) $ is a quasi-isomorphism for all $\scr{K} \in K $  if and only if $\uhom(C, p^{*}(f))$ is a quasi-isomorphism. Hence, $p_{*}$ all preserves weak equivalences and reflects weak equivalences. Since $p_{*}$ preserves fibrations, we can conclude that the adjunction is Quillen. 

The adjunction is a Quillen equivalence by \cite[Lemma 3.3]{erdal2019model} if, for all cofibrant contramodules $Q$, the unit of the adjunction $Q \lra p_{*}(p^{*}(Q))$ is a weak equivalence. Since $p_{*}(p^{*}(Q)) \cong Q$ the result follows.  
\end{proof}

\begin{proposition}
Let $L = \{(J(\scr{K})^{fib})^{cof} | \scr{K} \in K\}$. Then the right Bousfield localisation of  $ PMod(\underline{k})$ at $L$ exists and  $R_{L} PMod(\underline{k})$ is a right proper, combinatorial, stable, dg model category. Let  $f: \scr{F} \lra \scr{G}$ be a morphism of presheaves. Then, 
    \begin{enumerate}[label=(\roman*)]
    \item $f$ is a weak equivalence if and only if the induced map $\rdh(\scr{L}, f)$ is a quasi-isomorphism for all $\scr{L} \in L$,
    \item $f$ is a fibration if and only if $f(U): \scr{F}(U) \lra \scr{G}(U)$ is surjective for each open set $U \subset X$ and kernel of $f$ is a hypersheaf. 
\end{enumerate}

\end{proposition}

\begin{proposition}\label{sing qeq 2}
   The adjunction 

   \begin{equation*}
       - \ot_{\underline{k}} \scr{C}^{*} : R_{L}PMod(\underline{k}) \rightleftarrows R_{K}PCtmod(\scr{C}^{*}): J,
   \end{equation*}

   is a Quillen equivalence. 
\end{proposition}
\begin{proof}

Since $\underline{k} \lra \scr{C}^{*}$ is a local weak equivalence the adjunction between the adjunction between the local model structures

\begin{equation*}
     - \ot_{\underline{k}} \scr{C}^{*} : PMod(\underline{k}) \rightleftarrows PCtmod(\scr{C}^{*}): J
\end{equation*}
is a Quillen equivalence. The rest of the proof is similar to the proof of  Proposition \ref{proposition gen eq}.

\end{proof}
Combining Proposition  \ref{sing qeq 1} and Proposition  \ref{sing qeq 2} we obtain the following zig-zag of Quillen equivalences. 

\begin{theorem}\label{last theorem}
Let $X$ be a connected and locally contractible topological space and $k$ a field of characteristic $0$. Then the adjunctions 

 \[
\begin{tikzcd}[row sep=2em]
C^{*}(X)-Ctmod \arrow[r, shift left, "p^{*}"]  &  R_{K} PCtMod(\scr{C}^{*}) \arrow[l, shift left, "p_{*}" ]  \arrow[r, shift right, swap, "J" ] & R_{L}PCtMod(\underline{k}) \arrow[l, shift right,swap, "-\ot_{\underline{k}} \scr{C}^{*}" ]
\end{tikzcd}
\]   
 form a zig-zag dg Quillen equivalences between the dg model categories $C^{*}(X)-Ctmod$ and $R_{L}PMod(\underline{k})$. 
 
\end{theorem}

In light of Theorem \ref{last theorem},  it follows that the homotopy category of $R_{L}PMod(\underline{k})$ is equivalent to several categories, which can be thought of as forming a higher Riemann-Hilbert correspondence.  Let $D_{lf}(\underline{k}^{+})$ denote the derived category of clc sheaves over the constant sheaf $\underline{k}^{+}$, as defined in Definition \ref{clc def}.  The derived category of  clc sheaves are closely related to various notions of "$\infty$-local systems". Let $Sing(X)$ denote the singular simplicial complex of $X$. In \cite{holstein2015morita} the dg category of $\infty$-local systems is defined as  the dg category given by the cotensor action of the simplicial set $Sing(X)$ on the dg category $k-Mod$. Let $Loc_{\infty}(X)$ denote the dg category of $\infty$-local systems.  This is a pretriangulated dg category, and there exists an equivalence:

\begin{equation*}
        H^{0}(Loc_{\infty}(X)) \simeq D_{lf}(\underline{k}^{+}), 
\end{equation*}

(see, \cite[Theorem 12]{holstein2015moritaclc}).  In this sense, we can view the dg category of $\infty$-local systems as a dg enhancement of $D_{lf}(\underline{k}^{+})$.  The notion of an $\infty$-local systems was first defined in \cite{block2014higher} as a combinatorial refinement of the classical notion of a local system. As noted in \cite[Remark 3.7]{holstein2014properness} this is equivalent to the definition used in \cite{holstein2015morita}.  Alternatively, one can interpret  an $\infty$-local system as an $\infty$-representation of the $\infty$-groupoid $Sing(X)$, in particular, a higher analogue of a representation of the fundamental group.

More precisely, let $N_{dg}: dgCat_{k} \lra \mathbf{sSet}$ denote the dg nerve functor \cite[See Section 1.3.1 ]{lurie2009higher} and let $Fun$ denote the internal hom functor for the category of simplicial sets.  Then, we can define $\infty$-local systems as a quasi-category by 
\begin{equation}\label{quasi-infty eq}
     Fun(Sing(X),N_{dg}(k-Mod) ).
\end{equation}

It follows from \cite[section 4]{rivera2020colimit} that the dg nerve of the dg catgeory of $\infty$-local system as defined in \cite{holstein2015morita} is equivalent as a quasi-category to (\ref{quasi-infty eq}).

Now, assume that $X$ is path connected with base point $x$.  Let  $C_{*}(L_{x}(X))$ denote the dg algebra of singular chains on the Moore loop space of $X$ with base point  $x$. Then, 
\begin{equation*}
    H^{0}(Loc_{\infty}(X)) \simeq D(C_{*}(L_{x}(X))),
\end{equation*}
by \cite[Theorem 26]{holstein2015morita}. Let $C_{*}(X)$ denote the dg coalgebra of normalised the singular chains with values in $k-Mod$.  Let $\Omega $ denote the cobar construction. Then, 

\begin{align*}
 D^{co}(C_{*}(X))  &\simeq       D(\Omega (C_{*}(X))\\
                                          &\simeq D(C_{*}(L_{x}(X))),
\end{align*}

where first equivalence follows from Koszul duality \cite[Theorem 6.5 (a)]{positselski2011two} and, the second follows form \cite[Corollary 4.8]{chuang2021homotopy}. The second equivalence is a generalisation of the classical result by Adams, which states that for a simply connected space $\Omega (C_{*}(X)$ and $C_{*}(L_{x}(X))$  are quasi-isomorphic (see,  \cite{adams1956cobar}).

By the co-contra correspondence \cite[Threoem 5.2]{positselski2011two} the category of dg $C_{*}(X)$-comodules is Quillen equivalent to the category of dg $C_{*}(X)$-contramodules. Hence, by Theorem \ref{last theorem}

\begin{equation*}
    D^{co}(C_{*}(X))  \simeq  Ho(R_{L}PMod(\underline{k})).
\end{equation*}
We summarise   this discussion in the following corollary.

\begin{corollary}\label{last coro}
    Let $X$ be a path connected and locally contractible topological space with a base point $x$. Then, 
 \begin{enumerate}[label=(\roman*)]
\item $D_{lf}(\underline{k})$,
\item $D^{ct}(C^{*}(X))$,
\item $Ho(R_{L}PMod(\underline{k}))$,
\item $H^{0}(Loc_{\infty}(X))$,
\item $D(C_{*}(L_{x}(X)))$,
\end{enumerate}
are equivalent as categories. 
\end{corollary}

\cite{chuang2021maurer}
\bibliographystyle{plain}
\bibliography{mybib}

\end{document}